\documentclass[a4paper,reqno]{amsart}

\usepackage[utf8]{inputenc}
\usepackage{xr-hyper}
\usepackage[colorlinks=false]{hyperref}
\usepackage{pdfsync}
\usepackage{verbatim}
\usepackage{amsthm,amsfonts,amssymb,mathtools}
\usepackage{xcolor}
\usepackage{graphicx}
\usepackage{tikz-cd}
\usepackage{enumitem}
\usepackage{stmaryrd}
\usepackage{bbm}
\usepackage{relsize}

\newcommand{\E}{\mathcal{E}}
\newcommand{\F}{\mathcal{F}}
\newcommand{\G}{\mathcal{G}}

\newcommand{\I}{\mathcal{I}}

\renewcommand{\L}{\mathcal{L}}
\newcommand{\M}{\mathcal{M}}
\newcommand{\N}{\mathcal{N}}
\renewcommand{\O}{\mathcal{O}}
\renewcommand{\P}{\mathcal{P}}

\newcommand{\f}{\mathfrak{f}}

\newcommand{\CC}{\mathbb{C}}

\newcommand{\FF}{\mathbb{F}}
\newcommand{\GG}{\mathbb{G}}

\newcommand{\NN}{\mathbb{N}}

\newcommand{\ZZ}{\mathbb{Z}}

\newcommand{\coker}{\operatorname{coker}}

\newcommand{\Der}{\operatorname{Der}}

\newcommand{\End}{\operatorname{End}}

\newcommand{\Ext}{\operatorname{Ext}}

\newcommand{\Frob}{\operatorname{Frob}}

\newcommand{\GL}{\operatorname{GL}}

\newcommand{\Gr}{\operatorname{Gr}}
\newcommand{\gr}{\operatorname{gr}}
\newcommand{\Grp}{\operatorname{Grp}}

\newcommand{\Hom}{\operatorname{Hom}}
\newcommand{\id}{\operatorname{id}}
\newcommand{\im}{\operatorname{im}}

\newcommand{\Lie}{\operatorname{Lie}}
\newcommand{\Mod}{\operatorname{Mod}}

\newcommand{\ord}{\operatorname{ord}}

\newcommand{\pr}{\operatorname{pr}}

\newcommand{\Prim}{\operatorname{Prim}}

\newcommand{\Res}{\operatorname{Res}}

\newcommand{\rk}{\operatorname{rk}}

\renewcommand{\span}{\operatorname{span}}
\newcommand{\Spec}{\operatorname{Spec}}

\newcommand{\Sym}{\operatorname{Sym}}

\newcommand{\tors}[1]{#1_{\operatorname{tors}}}

\renewcommand{\bar}{\overline}
\newcommand{\colim}{\varinjlim}
\newcommand{\congr}{\equiv}

\renewcommand{\epsilon}{\varepsilon}

\renewcommand{\hat}{\widehat}

\renewcommand{\implies}{\Rightarrow}
\newcommand{\into}{\hookrightarrow}

\newcommand{\onto}{\twoheadrightarrow}
\newcommand{\isom}{\cong}
\renewcommand{\lim}{\varprojlim}

\newcommand{\longiff}{\Longleftrightarrow}

\newcommand{\longinto}{\lhook\joinrel\longrightarrow}
\newcommand{\longonto}{\relbar\joinrel\twoheadrightarrow}
\newcommand{\longto}{\longrightarrow}

\newcommand{\minus}{\smallsetminus}
\renewcommand{\mod}{\bmod}
\renewcommand{\phi}{\varphi}

\newcommand{\xto}[1]{\xrightarrow{#1}}

\newcommand{\dual}[1]{#1^{\vee}}
\newcommand{\Ga}{\mathbb{G}_a}
\newcommand{\Gm}{\mathbb{G}_m}

\newcommand{\isoto}{\xrightarrow{\ \raisebox{-2pt}[0pt][0pt]{\ensuremath{\sim}}\ }}

\newcommand{\pair}[1]{\langle{#1}\rangle}
\newcommand{\pullback}{\mathlarger{\mathlarger{\mathlarger{\mathlarger{\cdot\!\!\lrcorner}}}}}

\newcommand{\pushout}{\mathlarger{\mathlarger{\mathlarger{\mathlarger{\raisebox{-2.5pt}{\ensuremath{\ulcorner}}\!\!\cdot}}}}}

\DeclareRobustCommand{\skipTOCentry}[5]{}

\newtheorem{thm}{Theorem}[section]
\newtheorem{conj}[thm]{Conjecture}
\newtheorem{cor}[thm]{Corollary}
\newtheorem{lemma}[thm]{Lemma}
\newtheorem{prop}[thm]{Proposition}
\theoremstyle{definition}
\newtheorem{defn}[thm]{Definition}
\newtheorem{ex}[thm]{Example}
\newtheorem{rem}[thm]{Remark}

\numberwithin{equation}{section}

\newcommand{\Fp}{\underline{\mathbb{F}_p}}
\newcommand{\Fq}{\underline{\mathbb{F}_q}}

\renewcommand{\Gr}[1]{\operatorname{Gr}_{#1}}
\renewcommand{\gr}[1]{\operatorname{gr}_{#1}}
\renewcommand{\Grp}[1]{\operatorname{Gr}^{+}_{#1}}
\newcommand{\Grpf}[1]{\operatorname{Gr}^{\boxplus}_{#1}}
\newcommand{\grp}[1]{\operatorname{gr}^{+}_{#1}}
\newcommand{\Gra}[2]{\mathbb{F}_{#1}\mbox{-}\operatorname{Gr}_{#2}}
\newcommand{\gra}[2]{\mathbb{F}_{#1}\mbox{-}\operatorname{gr}_{#2}}
\newcommand{\Grpa}[2]{\mathbb{F}_{#1}\mbox{-}\operatorname{Gr}^{+}_{#2}}
\newcommand{\Grpaf}[2]{\mathbb{F}_{#1}\mbox{-}\operatorname{Gr}^{\boxplus}_{#2}}

\newcommand{\Grpaffp}[2]{\mathbb{F}_{#1}\mbox{-}\operatorname{Gr}^{\boxplus}_{#2,\operatorname{f.p.}}}
\newcommand{\grpa}[2]{\mathbb{F}_{#1}\mbox{-}\operatorname{gr}^{+}_{#2}}
\newcommand{\Grb}[2]{\mathbb{F}_{#1}\mbox{-}\operatorname{Gr}^{+,\operatorname{b}}_{#2}}

\newcommand{\Grbfp}[2]{\mathbb{F}_{#1}\mbox{-}\operatorname{Gr}^{+,\operatorname{b}}_{#2,\operatorname{f.p.}}}
\newcommand{\Grbffp}[2]{\mathbb{F}_{#1}\mbox{-}\operatorname{Gr}^{\boxplus,\operatorname{b}}_{#2,\operatorname{f.p.}}}
\newcommand{\grb}[2]{\mathbb{F}_{#1}\mbox{-}\operatorname{gr}^{+,\operatorname{b}}_{#2}}
\newcommand{\Grplike}{\operatorname{Grp}}
\newcommand{\Hopf}[1]{\operatorname{Hopf}_{#1}}
\newcommand{\hopf}[1]{\operatorname{hopf}_{#1}}
\newcommand{\Hopfa}[2]{\mathbb{F}_{#1}\mbox{-}\operatorname{Hopf}_{#2}}
\newcommand{\hopfa}[2]{\mathbb{F}_{#1}\mbox{-}\operatorname{hopf}_{#2}}
\newcommand{\Mss}{M_{\operatorname{ss}}}
\newcommand{\fss}{f_{\operatorname{ss}}}
\newcommand{\Mnil}{M_{\operatorname{nil}}}
\newcommand{\fnil}{f_{\operatorname{nil}}}
\renewcommand{\Mod}{\operatorname{-Mod}}
\newcommand{\FMod}[1]{\operatorname{Sht}_{#1}}
\newcommand{\Fmod}[1]{\operatorname{sht}_{#1}}
\newcommand{\Sht}[2]{\mathbb{F}_{#1}\mbox{-}\operatorname{Sht}_{#2}}
\newcommand{\Shtfg}[2]{\mathbb{F}_{#1}\mbox{-}\operatorname{Sht}_{#2,\operatorname{f.g.}}}
\newcommand{\sht}[2]{\mathbb{F}_{#1}\mbox{-}\operatorname{sht}_{#2}}

\begin{document}

\title{Group schemes with $\mathbb{F}_q$-action}
\author{Thomas Poguntke}
\address{Hausdorff Center for Mathematics \\ Endenicher Allee 62, 53115 Bonn, Germany}
\email{thomas.poguntke@hcm.uni-bonn.de}


\begin{abstract}
{Via a construction due to V. Drinfel'd, we prove an equivalence of categories, generalizing the equivalence between commutative flat group schemes in characteristic $p$ with trivial Verschiebung and their Dieudonn\'e modules to group schemes with $\FF_q$-action.}
\end{abstract}

\date{\today}
\maketitle

{\hypersetup{linkbordercolor=white} \tableofcontents}

\section{Introduction}\label{sec1}

Let $p$ be a prime and let $k$ be a field of characteristic $p$. Denote by $\Grp k$ the category of affine commutative group schemes over $k$ which can be embedded into $\Ga^N$ for some set $N$. We assign to $G\in\Grp k$ its Dieudonn\'e $\FF_p$-module $\M(G) = \Hom_{\Grp k}(G,\Ga)$, with the obvious left module structure over $\End_{\Grp k}(\Ga) \isom k[F]$, the non-commutative polynomial ring with
\[F\lambda = \lambda^pF \mbox{ for } \lambda\in k.\]
These Dieudonn\'e modules completely classify group schemes of the above type, as follows.

\begin{thm}[\cite{DG}, IV, $\mathsection$3, 6.7]\label{mainthmDG}
The contravariant functor $\M$ defines an exact anti-equivalence of categories
\begin{equation}\label{Mpoverk}
\M\!: \Grp k \longto k[F]\Mod.
\end{equation}
Under this duality, algebraic group schemes correspond to finitely generated $k[F]$-modules, and finite group schemes to finite-dimensional $k$-vector spaces.
\end{thm}

The above result allows us to describe the structure of our category over a perfect field, and its simple objects if the base is algebraically closed.

\begin{thm}[\cite{DG}, IV, $\mathsection$3, 6.9]\label{structurethm}
Let $k$ be a perfect field. Then $G\in\Grp k$ is algebraic if and only if it can be written as a product
\[G \isom \Ga^n \times \pi_0(G) \times H,\]
where $n\in\NN$, $H$ is a finite product of group schemes of the form $\alpha_{p^s}$, and $\pi_0(G)$ is an \'etale sheaf of finite $\FF_p$-vector spaces. If $k$ is algebraically closed, then
\[\pi_0(G) \isom (\Fp)^m,\ m\in\NN.\]
\end{thm}

On the other hand, let $S$ be a scheme of characteristic $p$. Consider the category $\operatorname{gr}^{+\vee}_{S}$ of locally finitely presented flat group schemes over $S$ of height $\leq 1$ (i.e. killed by their Frobenius). Let $p\mbox{-}\operatorname{Lie}_{S}$ denote the category of finite locally free $\O_S$-$p$-Lie algebras. Then we have the following classification theorem, similar to the above.

\begin{thm}[\cite{SGA}, Remark 7.5]\label{mainthmSGA}
The covariant functor
\[\L\!: \operatorname{gr}^{+\vee}_{S} \longto p\mbox{-}\operatorname{Lie}_{S},\ G \longmapsto \Lie(G),\]
defines an equivalence of categories.
\end{thm}

Two of our main results generalize Theorem \ref{mainthmDG}, resp. reduce to Theorem \ref{mainthmSGA} via Cartier duality (when ``$q=p$''). Moreover, we formulate two conjectures under which they unify.

Assume that $S$ is an $\FF_q$-scheme for some prime power $q=p^r$. Our group schemes $G$ are affine, commutative, flat over $S$ and carry an $\FF_q$-action. We require that locally on $S$, there is an embedding $G \into \Ga^N$ for some set $N$, which respects the $\FF_q$-actions.

The category of these group schemes will be denoted by $\Grpa qS$, and its full subcategory of finite group schemes of finite presentation is called $\grpa qS$.

On the other hand, we consider left $\O_S[F^r]$-modules, which are flat as $\O_S$-modules. They are called $\FF_q$-shtukas over $S$, and their category is denoted by $\Sht qS$. We write $\sht qS$ for the full subcategory of $\Sht qS$ of locally free modules of finite rank over $\O_S$.

We study the following generalization of the contravariant functor \eqref{Mpoverk},
\[\M_q = \M\!: \grpa qS \longto \sht qS,\ G \longmapsto \Hom_{\Grpa qS}(G,\Ga).\]
We also explain the construction of a functor in the other direction,
\[\G_q = \G\!: \sht qS \longto \grpa qS,\]
which is fully faithful and left-adjoint to $\M$. However, $\G_q$ does not define an equivalence of categories for $q \neq p$. Rather, we describe a full subcategory $\grb qS$ of \textit{balanced} group schemes in $\grpa qS$, and prove that it is the essential image of $\G$.

Namely, let $G = \Spec(B_G)\in\grpa qS$. Then we show that the space of primitive elements in the affine algebra of $G$ decomposes into eigenspaces for the $\FF_q^{\times}$-action as
\begin{equation}\label{primdecomp}
\Prim(B_G) = \bigoplus_{s=0}^{r-1} \Prim_{p^s}(B_G).
\end{equation}
Now $G$ is called balanced, if the $p$-Frobenii $\Prim_{p^t}(B_G) \to \Prim_{p^{t+1}}(B_G),\ x\mapsto x^p$, are bijective for all $0\leq t < r-1$. Note that when $q=p$, we recover $\grb pS = \grp S$.

\begin{thm}\label{finiteversion}
The functor $\G\!: \sht qS \to \grb qS$ defines an exact anti-equivalence of categories with quasi-inverse $\M$.
\end{thm}

Our definition of the balanced subcategory of $\grpa qS$ is inspired by Raynaud's paper \cite{Ray}. He considers finite commutative group schemes $G$ with an action of $\FF_q$, and the decomposition of the augmentation ideal into eigenspaces for the $\FF_q^{\times}$-action,
\[I_G = \bigoplus_{j=1}^{q-1} I_j,\]
similarly to \eqref{primdecomp}. Note that all summands $I_j$ are finite locally free $\O_S$-modules. Raynaud imposes the condition that $\rk(I_j) = 1$, for all $j$.

We define a group scheme $G\in\grpa qS$ to be \textit{quasi-balanced} if $\rk(I_j)$ is the same for all $j$. This turns out to be almost the same as being balanced; in particular, Raynaud's condition implies the balance property. The following theorem is our second main result.

\begin{thm}
Every $G\in\grb qS$ is quasi-balanced. For $q\neq 4$, the converse holds.
\end{thm}

Finally, we consider the question whether $\M\!: \Grb qS \to \Sht qS$ defines an equivalence of categories in general. In order to make sense of this, we have to assume the following.

\begin{conj}
For any $G\in\Grpa qS$, the $\O_S$-module $\M(G)$ is flat.
\end{conj}

Recall from above that it holds for finite $G$. Moreover assume the following key statement.

\begin{conj}
For $G \in\Grpa qS$ locally of finite presentation, locally on $S$, there exists an embedding $G \into \Ga^N$, with $N\in\NN$, such that the morphism $\M(\Ga^N) \to \M(G)$ is surjective.
\end{conj}

This follows from Theorem \ref{mainthmDG} if $S=\Spec(k)$ is a point. Moreover, it is true in the finite case. We obtain the conditional result that $\M$ is an equivalence if we restrict ourselves to finitely presented group schemes and finitely generated $\O_S[F^r]$-modules, respectively.

In particular, if $S=\Spec(k)$, we obtain the following generalization of Theorem \ref{structurethm}.

\begin{thm}
Let $k$ be a perfect field. Then a balanced group scheme $G\in\Grb qk$ is algebraic if and only if it is isomorphic to a product
\[G \isom \Ga^n \times \pi_0(G) \times H,\]
with $n\in\NN$ and $H$ a product of group schemes of the form $\alpha_{q^s}$, and where $\pi_0(G)$ is an \'etale sheaf of finite $\FF_q$-vector spaces. If $k$ is algebraically closed, $\pi_0(G) \isom (\Fq)^m$, for some $m\in\NN$.
\end{thm}

Theorem \ref{finiteversion} has an interesting history. In his article \cite{Drin}, $\mathsection$2, Drinfel'd defines the functor $\G\!: \sht qS \to \grpa qS$ and shows that it is fully faithful and exact. Furthermore, he proves that the \'etale group schemes in $\grpa qS$ lie in the essential image of $\G$.

In Laumon's book \cite{Lau}, App. B, he claims that $\sht qk$ is anti-equivalent to $\grpa qk$, where $k$ is a perfect field of characteristic $p$. However, $\grpa qk$ is not an abelian category for $q\neq p$, and $\alpha_p$ is of $\FF_q$-additive type but not balanced. This error was pointed out to us by Hartl. Laumon's argument is sufficiently detailed to locate the mistake in his reasoning.

In \cite{Ta}, Proposition 1.7, Taguchi gives a (rather brief) proof of Theorem \ref{finiteversion}. However, he describes $\grb qS$ by a condition on the order of the group schemes, which precludes a generalization to the category $\Grpa qS$ as above.

Abrashkin \cite{Abr} considers a category $\operatorname{DGr}^*(\FF_q)_S$, based on a definition of Faltings \cite{Fal}. Roughly, the $\FF_q$-action on $G \in\grpa qS$ is \textit{strict}, if $G$ has a deformation $G^{\flat}$ (which is then universal with respect to its $\FF_q$-action) such that $\FF_q$ acts via scalar multiplication on the associated representative of the cotangent complex.

In $\mathsection$2.3, Abrashkin constructs an equivalence of categories $D_q\!: \sht qS \isoto \operatorname{DGr}^*(\FF_q)_S$. Moreover, in $\mathsection$2.3.2., he shows that a group scheme carrying a strict $\FF_q$-action is balanced. Hence, the obvious functor $\operatorname{DGr}^*(\FF_q)_S \to \grb qS$ is well-defined, and it is clear from the constructions that the following diagram commutes.
\[\begin{tikzcd}
  \sht qS \ar[r, "{D_q}"] \ar[dr, "{\G_q}"'] & \operatorname{DGr}^*(\FF_q)_S \ar[d] \\
  & \grb qS
\end{tikzcd}\]
The above equivalence of categories appears in Hartl-Singh \cite{SinH}, Theorem 5.2, at the torsion level of the function field analogue of the crystalline Dieudonn\'e theory for $p$-divisible groups they establish over a general base. This was one of the main motivations for our study. For further applications in this direction, see for example Hartl-Kim \cite{HK}, as well as the paper \cite{GL} by Genestier and V. Lafforgue, where Theorem \ref{finiteversion} appears as Proposition 0.3.

Let us briefly outline the structure of the paper. In $\mathsection$\ref{sec2}, we provide some basic theory of group schemes we need. Section $\mathsection$\ref{sec3} specializes to group schemes of additive type, and culminates in the proof of Theorem \ref{finiteversion} in the crucial case $q=p$. Some details are postponed to avoid repetition and streamline the argument.

In $\mathsection$\ref{sec4} and $\mathsection$\ref{sec5}, we define the categories $\sht qA$ and $\grpa qA$, respectively, and study their internal structure. Section $\mathsection$\ref{sec6} is concerned with the construction of the functors $\M$ and $\G$, a more detailed analysis of their properties, and the proof of Theorem \ref{finiteversion}.

In $\mathsection$\ref{sec7}, we introduce quasi-balanced group schemes, and compare the two balance conditions. Finally, $\mathsection$\ref{sec8} concerns the question what we can still say in the case of infinite group schemes.

\addtocontents{toc}{\skipTOCentry}
\section*{Acknowledgements}
I would like to thank M. Rapoport for suggesting this topic, for his constant interest and his suggestions. I thank U. Hartl for granting us access to an early draft of the paper \cite{SinH} with Singh. I am indebted to Lisa Sauermann for her contribution to the penultimate section. I would also like to thank M. Raynaud for interesting correspondences on group schemes with trivial Verschiebung. G. Faltings pointed out some mistakes in a preliminary version of the manuscript; for this I am grateful. Finally, I would like to thank the anonymous referee for very helpful suggestions and comments.

\section{Preliminaries on group schemes}\label{sec2}

Let $p$ be a prime number, and $S$ a scheme of characteristic $p$.

\begin{defn}
For an $S$-scheme $X$, denote by $\Frob_X\!: X \to X$ its Frobenius endomorphism. Let $X^{(p)} = X \times_{S,\Frob_S} S$. The relative Frobenius $F_X\!: X \to X^{(p)}$ of $X$ is defined by the following diagram with cartesian square.
\[\begin{tikzcd}
  X \ar[dr, dashed, "{F_X}"] \ar[ddr, bend right, "{\Frob_X}"'] \ar[drr, bend left] & & \\
  & X^{(p)} \ar[r] \ar[d] \ar[dr, phantom, "\pullback" very near start] & S \ar[d, "{\Frob_S}"] \\
  & X \ar[r] & S.
\end{tikzcd}\]
In particular, $F_X$ is a morphism of $S$-schemes.
\end{defn}

\begin{defn}
We denote by $\Gr S$ the category of affine commutative flat group schemes over $S$, and its full subcategory of finite group schemes locally of finite presentation by $\gr S$.
\end{defn}

\noindent\textbf{Convention:} All of our considerations take place locally on $S$. To emphasize when we assume $S=\Spec A$, we will write $\Gr A = \Gr S$. We also fix $G = \Spec B_G$ as a notation.

\begin{defn}
We write $\Hopf A$, resp. $\hopf A$, for the opposite category of $\Gr A$, resp. $\gr A$.
\end{defn}

\begin{defn}
Let $G = \Spec B_G \in \Gr A$. Consider the symmetrization morphism
\begin{equation}\label{symmetrization}
s\!: B_G^{\otimes p} \longto \operatorname{TS}^p(B_G),\ x_1 \otimes \ldots \otimes x_p \longmapsto \sum_{\pi\in S_p} x_{\pi(1)} \otimes \ldots \otimes x_{\pi(p)},
\end{equation}
where $\operatorname{TS}^p(B_G) := (B_G^{\otimes p})^{S_p}$. Since $G$ is flat, $x \mapsto x^{\otimes p}$ induces an isomorphism
\[\sigma_p^* B_G := B_G \otimes_{A,\sigma_p} A \isoto \operatorname{TS}^p(B_G)/s(B_G^{\otimes p}),\]
see \cite{DG}, IV, $\mathsection$3, 4.1. Here, $\sigma_p$ denotes the Frobenius of $A$, and so we have $G^{(p)} = \Spec(\sigma_p^* B_G)$, which is by the above a closed subscheme of $\operatorname{S}^pG := \Spec(\operatorname{TS}^p(B_G))$. The Verschiebung of $G$ is then defined as the composition
\[V_G\!: G^{(p)} \longinto \operatorname{S}^pG \xto{\operatorname{mult}} G,\]
where $\operatorname{mult}$ is the $p$-fold multiplication on $G$, which factors over $\operatorname{S}^pG$, since $G$ is commutative.
\end{defn}

\begin{rem}\label{FVpVF}
We have $F_G \circ V_G = p\cdot\id_G$, and $V_G \circ F_G = p\cdot\id_{G^{(p)}}$, by \cite{DG}, IV, $\mathsection$3, 4.6. On affine algebras, $V_G^*$ acts by taking $p$-th ``copowers''. In this sense, it is dual to the (relative) Frobenius, which we make precise below. The name (German for ``shift'') comes from the Verschiebung on Witt (co-)vectors, where it acts as an index shift (cf. \cite{Fon}, III, $\mathsection$3.1).
\end{rem}

\begin{defn}
For $G\in\Gr A$, let $\eta\!: B_G \to A$ be the augmentation -- or counit --  of $B_G$, given by the unit section of $G$. The augmentation ideal of $G$ is defined by $I_G = \ker(\eta)$.
\end{defn}

\begin{rem}
The short exact sequence
\[0 \longto I_G \longto B_G \xto{\ \eta\ } A \longto 0\]
is split on the right by the unit $\epsilon\!: A \to B_G$ of $B_G$, so that in fact $B_G = A \oplus I_G$. In particular, the $A$-module $I_G$ is flat.
\end{rem}

\begin{defn}
Let $G\in\Gr A$. The space of primitive elements in $B_G$ is defined by
\[\Prim(B_G) := \{x\in I_G \mid \Delta(x) = x\otimes 1 + 1\otimes x\},\]
where $\Delta$ is the comultiplication on $B_G$, i.e. the map induced by the multiplication of $G$. The subgroup of group-like elements of $B_G$ is defined by
\[\Grplike(B_G) := \{x\in B_G \mid \Delta(x) = x \otimes x,\ \eta(x) = 1\},\]
where $\eta\!: B_G \to A$ is the counit of $B_G$.
\end{defn}

\begin{ex}
The group structure on $\Ga = \Spec A[x]$ is defined by $x\in\Prim(B_{\Ga})$. Hence
\[\Prim(B_G) \isom \Hom(G,\Ga),\]
for $G\in\Gr A$, by the universal property of the polynomial algebra.
\end{ex}

\begin{rem}
Let $G\in\Gr A$ and $x \in I_G$. Then
\[\Delta(x) \congr x \otimes 1 + 1 \otimes x \mod I_G \otimes I_G,\]
cf. \cite{Tate}, $\mathsection$2.3. This explains the name ``primitive element''.
\end{rem}

\begin{ex}\label{VofPrim}
Let $x\in B_G$. On $x \in \Prim(B_G)$, the Verschiebung vanishes,
\[V_G^*(x) = x \otimes 1 \otimes \ldots \otimes 1 + \ldots + 1 \otimes \ldots \otimes 1 \otimes x = \frac{1}{(p-1)!}s(x \otimes 1 \otimes \ldots \otimes 1) \congr 0,\]
where $s$ is the morphism \eqref{symmetrization}. On the other hand, if $x \in \Grplike(B_G)$, then
\[V_G^*(x) = x^{\otimes p} \isom x \otimes_{\sigma_p} 1\]
acts identically.
\end{ex}

\begin{prop}\label{primoftensor}
Let $G = \Spec B$ and $H = \Spec C$ be affine group schemes over $A$. Then the primitive elements
\[\Prim(B \otimes C) = \Prim(B) \otimes 1 + 1 \otimes \Prim(C)\]
are compatible with tensor products.
\end{prop}
\begin{proof}
We have to show that the isomorphism
\[\Prim(B) \times \Prim(C) \isoto \Hom(G,\Ga) \times \Hom(H,\Ga) \isoto \Hom(G \times H, \Ga) \isoto \Prim(B \otimes C)\]
is given by $(y,z) \mapsto y \otimes 1 + 1 \otimes z \in \Prim(B \otimes C)$. By definition, the image of $(y,z)$ in $\Hom(G \times H, \Ga)$ is induced by the unique Hopf algebra morphism $f\!: A[x] \to B \otimes C$ so that
\[\begin{tikzcd}[column sep=tiny]
  & & x \!\!\! \ar[dll, mapsto] \ar[r, phantom, "{\in}"] & \!\!\! A[x] \ar[dll] \ar[d, "{f}", "{\exists !}"'] \ar[drr] \!\!\! & \ar[l, phantom, "{\ni}"] \!\!\! x \ar[drr, mapsto] & & \\
  y \ar[r, phantom, "{\in}"] & B & & B \otimes C \ar[ll, "{\pi_B}"] \ar[rr, "{\pi_C}"'] & & C & \ar[l, phantom, "{\ni}"] z
\end{tikzcd}\]
commutes. Here, the projection $\pi_B\!: B \otimes C \to B$ (and similarly $\pi_C\!: B \otimes C \to C$) is given by the injection $G \to G \times H,\ g \mapsto (g,0)$. Explicitly, $\pi_B(b \otimes c) = b\epsilon(\eta(c))$, where $\eta\!: C \to A$ is the counit and $\epsilon\!: A \to C$ the unit of $C$. But now indeed,
\[\pi_B(f(x)) = \pi_B(y\otimes 1 + 1 \otimes z) = y\epsilon(\eta(1)) + \epsilon(\eta(z)) = y,\]
because $z\in I_H = \ker(\eta)$. We have $\pi_C(f(x)) = z$ by the same argument.
\end{proof}

\begin{defn}
Let $G \in \Gr S$. The Cartier dual $\dual G = \underline\Hom(G,\Gm)$ of $G$ is defined by
\[\dual G\!: R \longmapsto \Hom_{\Gr R}(G \otimes R,\ \Gm \otimes R)\]
as a functor of points.
\end{defn}

\begin{ex}
The group scheme $\alpha_p = \Spec A[x]/(x^p)$, with $x\in\Prim(B_{\alpha_p})$, is self-dual,
\[\dual\alpha_p = \alpha_p,\]
cf. e.g. \cite{Pink}, $\mathsection$5, p.11. On the other hand, consider the constant group scheme
\[\Fp \isom \Spec A[x]/(x^p-x),\]
where $x$ is primitive. Its dual is given by the roots of unity, with group-like generator $T$,
\[\mu_p = \Spec A[T]/(T^p-1).\]
\end{ex}

\begin{lemma}[cf. \cite{DG}, II, $\mathsection$1, 2.10; \cite{Stix}, $\mathsection$3.2.2]\label{cartier}
Let $G \in \gr A$. Then $\dual G = \Spec \dual{B_G}$, where the dual Hopf algebra $\dual{B_G} = \Hom_A(B_G,A)$ carries as (co-)multiplication
\[\nabla_{\dual G}\!: \dual{B_G} \otimes \dual{B_G} \isom \Hom(B_G \otimes B_G, A) \xto{\Delta_G^*} \Hom(B_G,A),\]
\[\Delta_{\dual G}\!: \dual{B_G} \xto{\nabla_G^*} \Hom(B_G \otimes B_G, A) \isom \dual{B_G} \otimes \dual{B_G},\]
the transpose of the (co-)multiplication on $B_G$. Similarly for the unit and counit,
\[\epsilon_{\dual G}\!: A \isom \Hom_A(A,A) \xto{\eta_G^*} \dual B_G,\]
\[\eta_{\dual G}\!: \dual B_G \xto{\epsilon_G^*} \Hom_A(A,A) \isom A.\]
\end{lemma}

\begin{rem}
In the case of Lemma \ref{cartier}, it is easy to see that Frobenius and Verschiebung are dual to one another, as Cartier duality exchanges multiplication and comultiplication. The same is indeed true for any $G\in\Gr A$ by \cite{DG}, IV, $\mathsection$3, 4.9. Namely,
\begin{equation}\label{FVdual}
F_{\dual G} = \dual{(V_G)},
\end{equation}
and assuming $\dual G$ is represented by a flat group scheme over $A$, also
\[V_{\dual G} = \dual{(F_G)}.\]
\end{rem}

The following result is crucial for our main theorem in the finite case.

\begin{prop}\label{liedual}
For $G = \Spec B_G \in \gr A$, there are natural isomorphisms of $A$-modules
\[\Lie \dual G \isom \Hom_A(I/I^2, A) \isom \Der_A(\dual B,A) \isom \Prim B_G,\]
where $\dual B = \dual B_G$ and $I = \ker(\dual\eta\!: \dual B \to A)$ is the augmentation ideal of $\dual G$.
\end{prop}
\begin{proof}[Proof (cf. \cite{Fon}, I, $\mathsection$8.3 ff.)]
Let $A(\epsilon) = A[t]/(t^2)$ be the algebra of dual numbers and denote by $\pi\!: A(\epsilon) \onto A$ the projection. For $u \in \dual G(A(\epsilon))$, we have by definition
\[u \in \Lie \dual G = \ker \dual G(\pi) \longiff (\dual B \xto{u} A(\epsilon) \xto{\pi} A) = \dual\eta \longiff u(I) \subseteq \epsilon A(\epsilon).\]
In that case, we get $u(I^2) \subseteq \epsilon^2 A(\epsilon) = 0$, hence an element in the tangent space of $\dual G$,
\[\bar u\!: I/I^2 \longto A,\ \alpha \longmapsto \frac{u(\alpha)}{\epsilon}.\]
The second isomorphism is just the universal property (cf. \cite{Tate}, $\mathsection$2.11)
\[\Der_A(\dual B,A) \isom \Hom_{\dual B}(\Omega^1_{\dual B|A},A) \isom \Hom_{\dual B}(I/I^2 \otimes_A \dual B, A) \isom \Hom_A(I/I^2,A),\]
where the $\dual B$-module structure on $A$ is given by $\dual\eta$. Finally, consider the natural pairing
\[\pair{\mbox{-},\mbox{-}}\!: \dual B \times B_G \longto A,\ (\alpha,x) \longmapsto \alpha(x).\]
For $x\in B_G$, recalling Lemma \ref{cartier}, we have $x \in \Prim(B_G)$ if and only if
\[\pair{\alpha\beta,x} = (\alpha\otimes\beta)(\Delta(x)) = (\alpha\otimes\beta)(x \otimes 1 + 1 \otimes x) = \pair{\alpha,x}\dual\eta(\beta) + \dual\eta(\alpha)\pair{\beta,x},\]
that is to say $\pair{\mbox{-}, x} \in \Der_A(\dual B,A)$.
\end{proof}

\begin{rem}
Proposition \ref{liedual} will also allow us to dualize our theory, in the sense that
\[\Lie \underline\Hom(G,\Gm) = \Hom(G,\Ga),\]
for $G\in\gr A$. Therefore, Cartier duality reduces Theorem \ref{mainthm} to Theorem \ref{mainthmSGA}.
\end{rem}

\section{Group schemes of additive type}\label{sec3}

\begin{defn}
A group scheme $G \in \Gr S$ is of additive type if there exists a closed embedding of $G$ into $\Ga^N$ for some set $N$, locally on $S$. We define $\Grp A$, resp. $\grp A$, to be the full subcategory of $\Gr A$, resp. $\gr A$, of group schemes of additive type.
\end{defn}

\begin{thm}\label{embeddingsintoGa}
Let $G \in\Gr S$ be locally finitely presented. Then the following are equivalent.
\begin{enumerate}[label=\normalfont $(\roman*)$]
\itemsep3pt
\item $G \in \Grp S$.
\item $I_G = (\Prim B_G)$, i.e. $\Prim B_G$ generates $I_G$ as an ideal, locally on $S$.
\end{enumerate}
Moreover, the above conditions imply the following.
\begin{enumerate}[label=\normalfont $(\roman*)$]
\setcounter{enumi}{2}
\item $V_G = 0$.
\end{enumerate}
For finite $G \in \gr S$, all three conditions are equivalent.
\end{thm}
\begin{proof}[Proof (Raynaud)]
The equivalence of $(i)$ and $(ii)$ is clear. Indeed, an embedding $G \into \Ga^N$ is the same as a Hopf algebra epimorphism $A[x_1,\ldots, x_N] \onto B_G$, and the $x_n$ are primitive by definition. The implication "$(ii) \implies (iii)$" is settled by Example \ref{VofPrim}.

Now it remains to show that if $G\in\gr A$ and $A$ is a local ring, then $V_G = 0$ implies that there exists a closed embedding $G \into \Ga^N$ for some $N\in\NN$.

Consider the Cartier dual $\dual G$ of $G$, with affine algebra $\dual B = B_{\dual G}$ (cf. Lemma \ref{cartier}) and augmentation ideal $I = I_{\dual G}$. Since $B_G$ is finite and locally free, it is reflexive. Therefore, we have an inclusion on $R$-valued points as follows,
\begin{equation}\label{dualembedding}
G(R) \isom \Hom_{\Gr R}(\dual G \otimes R,\Gm \otimes R) \isom \Grplike(\dual B \otimes R) \subseteq (\dual B \otimes R)^{\times}.
\end{equation}
The functor $\Res_{\dual B/A}(\Gm \otimes \dual B)$ on the right-hand side is represented by a group scheme, because $\dual B$ is finite flat, and \eqref{dualembedding} defines a closed embedding $G \into \Res_{\dual B/A}(\Gm \otimes \dual B)$. Now, the counit of $\dual B$ induces a natural splitting
\[\eta\!: \Res_{\dual B/A}(\Gm \otimes \dual B) \longonto \Gm\]
of the natural inclusion $\Gm \into \Res_{\dual B/A}(\Gm \otimes \dual B)$. This yields a split short exact sequence
\[\begin{tikzcd}[row sep=0em, column sep=1.5em]
1 \ar[r] & \Gm \ar[r] & \Res_{\dual B/A}(\Gm \otimes \dual B) \ar[r] & 1+\I \ar[r] & 1, \\
 &  & \alpha \ar[r, mapsto] & \frac{\alpha}{\eta_R(\alpha)} \ar[r, phantom, "{\mathrlap{\mbox{(on } R\mbox{-points),}}}"] & \phantom{1}
\end{tikzcd}\]
where $(1+\I)(R) := 1 + \ker(\eta_R)$. Consider the kernel $H$ of the composition
\[G \longinto \Res_{\dual B/A}(\Gm \otimes \dual B) \isom \Gm \times (1+\I) \longonto 1+\I.\]
Then $H$ embeds into $\Gm$, so its fibres are of multiplicative type. But they are also killed by the Verschiebung, hence vanish (\cite{DG}, IV, $\mathsection$3, 4.11). By Nakayama, $H=0$.

Finally, since $F_{\dual G} = \dual{(V_G)} = 0$ by \eqref{FVdual}, we have $\I^p = 0$. Thus $1+\I$ is isomorphic via truncated $\exp$ and $\log$ to the finite free additive group $\I \isom \Ga^{\ord(G)-1}$.
\end{proof}

\begin{rem}
Over a field $A=k$, all three conditions in Theorem \ref{embeddingsintoGa} are equivalent, as shown in \cite{DG}, IV, $\mathsection$3, 6.6. We conjecture that this holds over an arbitrary base scheme $S$.
\end{rem}

\begin{rem}
If condition $(ii)$ in Theorem \ref{embeddingsintoGa} holds, the Hopf algebra $B_G$ is also called primitively generated, i.e. it is generated as an algebra by its primitive elements.
\end{rem}

\begin{ex}
The constant group scheme $\Fp$ over the $\FF_p$-algebra $A$ embeds into $\Ga$ via the projection
\[A[x] \onto A[x]/(x^p-x).\]
The same holds for the group schemes $\alpha_{p^s} = \Spec A[x]/(x^{p^s})$, for $s\in\NN$, since $x$ is primitive by definition. That is, they are all of additive type.
\end{ex}

We now compute the order of a finite group scheme $G$ of additive type. This is the essential step towards our main theorem.

\begin{prop}\label{orderofG}
Let $G \in \grp A$, and consider its dual $\dual B = \dual B_G$. Locally on $\Spec A$, there exists an algebra isomorphism
\[\dual B \isom A[t_1,\ldots,t_n]/(t_1^p,\ldots,t_n^p).\]
Moreover, the $A$-module $\Prim(B_G)$ is locally free, and the order of $G$ is given by
\[\ord(G) = p^{\rk(\Prim B_G)}.\]
\end{prop}
\begin{proof}[Proof (cf. \cite{SGA}, $\mathsection$7.4.3 and \cite{Lau}, Lemma B.3.14)]
Let $A$ be a local ring with residue field $k$. Let $I$ be the augmentation ideal of $\dual B$, which is then free of finite rank $d$. Write $I_k := I \otimes k$, and choose a basis $e_1,\ldots,e_d$ such that $e_{n+1},\ldots,e_d$ is a basis of $I_k^2$. Let $t_i \in I$ be a lift of $e_i$ for $1 \leq i \leq d$, so that $t_1,\ldots,t_d$ is an $A$-basis of $I$ by Nakayama.

Now consider the free $A$-submodule $M := \span_A(t_1,\ldots,t_n) \subseteq I$, and define
\[B' := \Sym(M)/(t^{\otimes p} \mid t \in M) \isom A[t_1,\ldots,t_n]/(t_1^p,\ldots,t_n^p).\]
Since $F_{\dual G} = \dual{(V_G)} = 0$, we have $I^p = 0$, and the canonical morphism
\[\psi\!: B' \longto \dual B,\ t_i \longmapsto t_i,\]
is well-defined. Surjectivity of $\psi$ is easy to check along the filtration
\[0=I^p \subseteq I^{p-1} \subseteq \ldots \subseteq I \subseteq \dual B.\]
We claim that in fact $\dim_k(B'\otimes k) = \dim_k(\dual B\otimes k)$, so that $\psi \otimes k$ is an isomorphism. To show this, we can assume $k$ to be perfect. Using $I_k^p=0$, we then know by \cite{DG}, III, $\mathsection$3, 6.3, that there is an algebra isomorphism
\[\dual B\otimes k \isom k[T_1,\ldots,T_N]/(T_1^p,\ldots,T_N^p).\]
But then in particular $N = \dim(I_k/I_k^2) = n$. Since both $\dual B$ and $B'$ are finite flat $A$-modules of finite presentation, $\psi$ is an isomorphism.

For the second part, it suffices by Proposition \ref{liedual} to show that $I/I^2$ is free. But $\psi^{-1}$ induces an isomorphism of $A$-modules
\[I/I^2 \isoto J/J^2 \isom M,\]
where $J$ denotes the augmentation ideal of $B'$.

Finally, Proposition \ref{liedual} then tells us that $\rk(\Prim B_G) = \rk(I/I^2) = n$, and thus indeed
\[\ord(G) = \rk(\dual B) = p^n = p^{\rk(\Prim B_G)},\]
as desired.
\end{proof}

\begin{rem}\label{primbasechange}
It is easy to see that for any $A$-algebra $R$, and any $G\in\Gr A$, we have a canonical map
\begin{equation}\label{primbc}
\Prim(B_G) \otimes_A R \longto \Prim(B_G \otimes_A R).
\end{equation}
Now let $G \in \grp A$, and assume that $R$ is the residue field at some point of $\Spec A$. To show that \eqref{primbc} is an isomorphism, we may assume $A$ to be local. But then $\eqref{primbc}$ is clearly injective, and both sides of \eqref{primbc} are finite $R$-modules of the same rank by Proposition \ref{orderofG}.
\end{rem}

\begin{defn}\label{addpolydef}
Let $A[F]$ be the non-commutative polynomial ring over $A$ with $F\lambda = \lambda^pF$ for any $\lambda\in A$. Note that $A[F] \isom \span_A(x^{p^e} \mid e \in \NN) \subseteq A[x]$ as $A[F]$-modules via $Fx = x^p$. The category of $A[F]$-modules, which are finite and locally free over $A$, is denoted by $\Fmod A$.
\end{defn}

\begin{prop}\label{addpoly}
Let $N$ be a set. The primitive elements in the affine algebra of $\Ga^N$ are
\[\Prim(A[x_n \mid n\in N]) = \span_A(x_n^{p^e} \mid n\in N,\ e \in \NN),\]
the space of additive polynomials in $A[x_n \mid n\in N]$. In other words,
\[\Hom(\Ga^N,\Ga) \isom A[F]^{\oplus N}\]
as $A[F]$-modules. In particular, there is a natural identification $\End(\Ga) = A[F]$.
\end{prop}
\begin{proof}
Of course, ``$\supseteq$'' holds by definition. Now it suffices to see that
\[\Prim(A[x]) \subseteq \span_A(x^{p^e} \mid e \in \NN),\]
since we may assume $N$ to be finite, and use induction over $\# N$ via Proposition \ref{primoftensor}. Let thus $z = \sum_{n\in\NN} \lambda_n x^n \in \Prim(A[x])$. Then
\[\Delta(z) = \sum_{n\in\NN} \lambda_n (x^n \otimes 1) + \sum_{n\in\NN} \lambda_n (1 \otimes x^n).\]
On the other hand, since $\Delta$ is an algebra morphism and $x$ is primitive,
\[\Delta(z) = \sum_{n\in\NN} \lambda_n (x \otimes 1 + 1 \otimes x)^n = \sum_{n\in\NN} \lambda_n \sum_{k\leq n} \binom nk (x^k \otimes x^{n-k}).\]
Comparing coefficients, we see that if $\lambda_n \neq 0$, then $\binom nk \congr 0 \mod p$ for all $0 < k < n$. But this implies that $n = p^e$ for some $e\in\NN$, cf. \cite{Fine}, Theorem 3.
\end{proof}

\begin{defn}
We denote the Dieudonn\'e $\FF_p$-module functor on $\grp A$ by
\[\M\!: \grp A \longto \Fmod A,\ G \longmapsto \Hom(G,\Ga).\]
Here, the $A[F]$-module structure on $\M(G)$ is given as in Definition \ref{addpolydef} by
\[Fx = x^p \in \Prim(B_G).\]
Equivalently, this is the obvious left module structure on $\Hom(G,\Ga)$ over $\End(\Ga) = A[F]$, cf. Proposition \ref{addpoly}. Conversely, let us define for $M\in\Fmod A$ the corresponding group scheme
\[\G(M) = \Spec(\Sym(M)/\f),\]
where $\f$ is the ideal $\f = (x^{\otimes p} - Fx \mid x\in M)$ in the symmetric algebra over $A$. The group structure on $\G(M)$ is defined by $\Prim(B_{\G(M)}) \supseteq M$, by extension to the whole algebra.
\end{defn}

\begin{rem}
The easy verification that the functors $\M$ and $\G$ are well-defined is a special case of Remark \ref{Mwelldef} and Remark \ref{imageofG}, respectively.
\end{rem}

\begin{ex}
Proposition \ref{addpoly} implies for the standard subgroup schemes of $\Ga$ that
\[\M(\alpha_{p^s}) \isom A[F]/(F^s), \mbox{ and } \M(\Fp) \isom A[F]/(F-1).\]
\end{ex}

\begin{thm}\label{mainthm}
The functor $\M$ defines an exact anti-equivalence of categories.
\end{thm}
\begin{proof}
Locally on $\Spec A$, choose a basis $x_1,\ldots,x_N$ of $M\in\Fmod A$. This yields the basis
\[\{\prod_{i=1}^N x_i^{e_i} \mid 0 \leq e_i < p\}\]
of $B_{\G(M)} = \Sym(M)/\f$, locally over $\Spec A$. Therefore, we obtain
\begin{equation}\label{orderofGM}
\ord \G(M) = p^{\rk M}.
\end{equation}
Now, it is not hard to see that the functor $\G\!: \Fmod A \to \grp A$ is left-adjoint to $\M$. We will give the details in Lemma \ref{adjunction}. For $G\in\grp A$ and $M \in \Fmod A$, consider the adjunction morphisms
\[u_G\!: G \longto \G(\M(G)), \mbox{ and } v_M\!: M \longto \M(\G(M)).\]
By construction, $v_M$ is the inclusion $M \into \Prim(B_{\G(M)})$. From \eqref{orderofGM} and Proposition \ref{orderofG}, as well as base change (Remark \ref{primbasechange}), we see that $v_M$ is an isomorphism.

Now consider the map $u_G^*$, which extends the identity on $\Prim(B_G)$ to a morphism
\[u_G^*\!: \Sym(\Prim(B_G))/(x^{\otimes p} - x^p \mid x\in\Prim(B_G)) \longto B_G.\]
Since $B_G$ is primitively generated, $u_G^*$ is surjective. By Proposition \ref{orderofG} as well as \eqref{orderofGM},
\[\ord \G(\M(G)) = p^{\rk(\Prim B_G)} = \ord G.\]
Thus $u_G^*$ is an epimorphism between finite locally free modules of the same rank, and hence bijective. Finally, we see that $\M$ is exact by (Lemma \ref{adjunction} and) additivity of the rank.
\end{proof}

\section{The category of \texorpdfstring{$\FF_q$}{Fq}-shtukas}\label{sec4}

Let $q=p^r$ be a power of the prime $p$, and assume that $S$ an $\FF_q$-scheme.

\begin{defn}
A finite $\FF_q$-shtuka over $S$ is a pair $(M,f)$, where $M$ is a finite locally free $\O_S$-module, and $f$ is a $q$-linear endomorphism of $M$. Equivalently, $f$ is a linearized map
\[f^{(q)}\!: \sigma_q^*M = M \otimes_{\O_S,\sigma_q} \O_S \longto M,\]
where $\sigma_q = \sigma_p^r$ is the Frobenius of $\O_S$. A morphism $\Phi\!: (M,f) \to (M',f')$ in the category $\sht qS$ of finite $\FF_q$-shtukas over $S$ is an $\O_S$-module morphism such that the diagram
\[\begin{tikzcd}
  M \ar[r, "{\Phi}"] \ar[d, "f"'] & M' \ar[d, "{f'}"] \\
  M \ar[r, "{\Phi}"] & M'
\end{tikzcd}\]
commutes. We shall write $\sht qA = \sht qS$, when $S=\Spec A$.
\end{defn}

\begin{rem}
Note that $(M,f) \in \sht pA$ is the same as the left $A[F]$-module $M$, defined by $Fx = f(x)$ for $x\in M$. Thus we recover $\Fmod A = \sht pA$.
\end{rem}

The definition comes from the following geometric example.

\begin{ex}
Let $X$ be a smooth projective geometrically irreducible curve over $\FF_q$. Then a (right) shtuka (or $F$-sheaf) of rank $d\in\NN$ over $S$ is a diagram
\begin{equation}\label{shtuka}
  \begin{tikzcd}[row sep=tiny]
    \L \ar[dr, "i"] & \\
    & \E \\
    \mathllap{(\id_X \times F_S)^*}\L \ar[ur, "{\tau}"] & 
  \end{tikzcd}
\end{equation}
with $\L$ and $\E$ locally free sheaves of $\O_{X \times S}$-modules of rank $d$, injective homomorphisms $\tau$ and $i$, and such that $\coker(\tau)$, resp. $\coker(i)$, is supported on the graph $\Gamma_{\alpha}$, resp. $\Gamma_{\beta}$, of some sections $\alpha,\beta\!: S \to X$ (called the zero, resp. the pole, of the shtuka).

Let $D\subseteq X$ be a finite subscheme away from the pole, i.e. $\beta(S) \subseteq X \minus D$. Then $i|_{D \times S}$ is an isomorphism. Setting $\L_D := \L|_{D \times S}$, we therefore obtain a morphism completing the restriction to $D \times S$ of diagram \eqref{shtuka},
\[f' = (i|_{D \times S})^{-1} \circ \tau|_{D \times S}\!: (\id_D \times F_S)^*\L_D \longto \L_D.\]
Denote by $\pi\!: D \times S \to S$ the projection, then $(\pi_*\L_D, f)$ is a finite $\FF_q$-shtuka over $\O_S$, where
\[f^{(q)}\!: F_S^*\pi_*\L_D \isom \pi_*(\id_D \times F_S)^*\L_D \xto{\pi_*f'} \pi_*\L_D.\]
Drinfel'd \cite{Drin} introduced $F$-sheaves in the proof of the Langlands conjecture for $\GL_2$ over a global field of characteristic $p$.
\end{ex}

Let us remark here a simple dichotomy in the category $\sht qk$, where $k$ is a perfect field. We will later use it to generalize Theorem \ref{structurethm}.

\begin{lemma}\label{ssnil}
Let $k$ be a perfect field. For $(M,f) \in \sht qk$, there is a unique decomposition
\[(M,f) = (\Mss,\fss) \oplus (\Mnil,\fnil)\]
such that $\fss = f|_{\Mss}$ is bijective and $\fnil = f|_{\Mnil}$ is nilpotent.
\end{lemma}
\begin{proof}[Proof (see \cite{Lau}, Lemma B.3.10)]
Since $k$ is perfect, let us identify $f^{(q)} = f$. Let
\[\Mss = \bigcap_{n\in\NN} \im(f^n), \mbox{ and } \Mnil = \bigcup_{n\in\NN} \ker(f^n).\]
Then there is some $N\in\NN$ with $\Mss = \im(f^N)$ and $\Mnil = \ker(f^N)$, so that in particular
\[\dim(M) = \dim(\Mss) + \dim(\Mnil).\]
Now suppose that $m\in\Mss \cap \Mnil$. Then we have $m = f^N(m')$ for some $m'\in M$, and we obtain $f^{2N}(m') = f^N(m) = 0$. But since $\ker(f^{2N}) = \ker(f^N)$, in fact $m = f^N(m') = 0$.
\end{proof}

\section{\texorpdfstring{$\FF_q$}{Fq}-actions on group schemes}\label{sec5}

\begin{defn}
Let $S$ be an $\FF_q$-scheme, and $G\in\Gr S$. An $\FF_q$-action on $G$ is a ring morphism
\[[-]_G\!: \FF_q \longto \End_{\Gr S}(G),\ \alpha \longmapsto [\alpha]_G.\]
A morphism of group schemes with $\FF_q$-action $\varphi\!: (G,[-]_G) \to (H,[-]_H)$ is a morphism of group schemes over $S$ such that the diagram
\begin{equation}\label{action}
  \begin{tikzcd}
    G \ar[r, "{\phi}"] \ar[d, "{[\alpha]_G}"'] & H \ar[d, "{[\alpha]_H}"] \\
    G \ar[r, "{\phi}"] & H
  \end{tikzcd}
\end{equation}
commutes for all $\alpha\in\FF_q$. When there is no ambiguity, we will just write $[\alpha]$ for the action.

\smallskip

We denote by $\Gra qS$ the category of group schemes in $\Gr S$, together with an $\FF_q$-action. For its objects, we will write $G$ instead of $(G,[-]_G)$. The full subcategory of $\Gra qS$ of finite group schemes is $\gra qS$. We replace $S$ by $A$ in the notation, when $S=\Spec A$. As before, we consider the dual categories $\Hopfa qA$, resp. $\hopfa qA$.
\end{defn}

\begin{ex}
When we consider $\Ga = \Spec A[x]$ as an object of $\Gra qA$, we mean that
\[[\alpha]^*x = \alpha x \mbox{ for all } \alpha\in\FF_q,\]
unless explicitly stated otherwise. The same extends to the group schemes $\alpha_{p^s} \subseteq \Ga$ and the constant group $\Fq \subseteq \Ga$, as well as any product of these groups.
\end{ex}

\begin{rem}\label{categorical}
Let $G,H \in\Gra qA$. Then $G \times H$ is endowed with the product $\FF_q$-action
\[[\alpha]_{G\times H} = [\alpha]_G \times [\alpha]_H.\]
Let $\phi,\psi\!: G \to H$ be morphisms in $\Gra qA$. Then the diagram
\[\begin{tikzcd}
  G \ar[r, "{\operatorname{diag}}"] \ar[d, "{[\alpha]_G}"'] & G \times G \ar[r, "{\phi \times \psi}"] \ar[d, "{[\alpha]_{G\times G}}"'] & H \times H \ar[r, "{\operatorname{mult}}"] \ar[d, "{[\alpha]_{H\times H}}"] & H \ar[d, "{[\alpha]_H}"] \\
  G \ar[r, "{\operatorname{diag}}"] & G \times G \ar[r, "{\phi \times \psi}"] & H \times H \ar[r, "{\operatorname{mult}}"] & H
\end{tikzcd}\]
commutes for all $\alpha\in\FF_q$, so that $\phi + \psi \in \Hom_{\Gra qA}(G,H)$ again.

Moreover, $\Gra qA$ is an $\FF_q$-linear category. Namely, $\Hom_{\Gra qA}(G,H)$ is given a vector space structure by the obvious actions (which agree by definition) of $\alpha\in\FF_q$ via
\[[\alpha]_G \in \End_{\Gra qA}(G), \mbox{ resp. } [\alpha]_H \in \End_{\Gra qA}(H).\]
Now consider an arbitrary fibre product diagram in $\Gra qA$ as in \eqref{pullbackaction} below. Then
\begin{equation}\label{pullbackaction}
  \begin{tikzcd}[row sep=tiny, column sep=small]
    G' \ar[rr] \ar[dd] \ar[dr, dashed] &  & H' \ar[dd] \ar[dr, "{[\alpha]_{H'}}"] &  \\
     & G' \ar[rr, crossing over] \ar[dr, phantom, "{\pullback}" very near start] &  & H' \ar[dd] \\
    G \ar[rr] \ar[dr, "{[\alpha]_G}"'] &  & H \ar[dr, "{\!\!\!\![\alpha]_H}"] &  \\
     & G \ar[rr] \ar[from=uu, crossing over] &  & H
  \end{tikzcd}
\end{equation}
defines a canonical $\FF_q$-action on $G'$, since all squares in \eqref{pullbackaction} commute, for all $\alpha\in\FF_q$. The dual construction yields an $\FF_q$-action for pushouts.

For example, if $A$ is a field, the $\FF_q$-action of $B_G$ descends to $\ker(\phi) = \Spec(B_G/\phi^*(I_H)B_G)$ (cf. \cite{Mil}, VII, Proposition 4.1), which is seen directly from the commutative diagram \eqref{action}.

Similarly, the $\FF_q$-action on $H$ restricts to $\coker(\phi) = \Spec C$, where
\[C = \{x \in B_H \mid \Delta(x) - 1 \otimes x \in \ker(\phi^*) \otimes B_H\},\]
see \cite{Fon}, I, $\mathsection$6.3. Indeed, by \eqref{action} again, for all $x\in C$, we have
\[\Delta([\alpha]^*x) - 1 \otimes [\alpha]^*x = ([\alpha]^* \otimes [\alpha]^*)(\Delta(x) - 1 \otimes x) \in \ker(\phi^*) \otimes B_H.\]
\end{rem}

\begin{defn}
Let $G = \Spec B_G \in \Gra qA$. The eigenspaces of the $\FF_q^{\times}$-action on the augmentation ideal $I=I_G$ are given by
\[I_j := I_j(G) := \{x\in I_G \mid [\alpha]^*x = \alpha^j x \mbox{ for all } \alpha\in\FF_q\},\]
for $0 < j < q$, identifying $\FF_q^{\times} \isom \ZZ/(q-1)$. We also set $\Prim_j(B_G) := \Prim(B_G) \cap I_j$.
\end{defn}

\begin{rem}\label{eigendecomp}
Since $\ord(\FF_q^{\times})$ is prime to $p$, the ideal $I_G$ decomposes into its eigenspaces as
\begin{equation}\label{decomp}
I_G = \bigoplus_{j=1}^{q-1} I_j.
\end{equation}
Indeed, we can write $\FF_q[\FF_q^{\times}] = \FF_q[X]/(X^{q-1}-1) = \bigoplus\limits_{j=1}^{q-1} \FF_q\chi_j$, with $\chi_j(\alpha) = \alpha^j$ for $\alpha\in\FF_q^{\times}$. This yields a system of orthogonal idempotents of $\End_A(I_G)$, cf. \cite{OoTa}, Lemma 2,
\[e_j = \frac{1}{q-1}\sum_{\alpha\in\FF_q^{\times}} \chi_j^{-1}(\alpha)[\alpha]^*,\ 0 < j < q.\]
Hence we obtain $\eqref{decomp}$, since $I_j = e_jI_G$. In particular, it follows that the $I_j$ are flat over $A$, as direct summands of the flat $A$-module $I$. The analogous statements hold for the $\Prim_j(B_G)$, if $\Prim(B_G)$ is flat, noting that $\Prim(B_G)$ is stable under $[\alpha]^*\in\End_{\Hopf A}(B_G)$, $\alpha\in\FF_q$.
\end{rem}

\begin{defn}
Let $G\in\Gra qS$. We say that $G$ is of $\FF_q$-additive type, if locally on $S$, there exists an $\FF_q$-equivariant closed embedding
\[G \longinto \Ga^N \mbox{ for some set } N.\]
The full subcategory of $\Gra qA$, resp. $\gra qA$, of group schemes $G$ of $\FF_q$-additive type is denoted by $\Grpa qA$, resp. $\grpa qA$. For $q = p$, we drop $\FF_q$ from the notation, as before.
\end{defn}

\begin{rem}\label{addtypeiffprimgen}
Let $G\in\Gra qS$ be locally of finite presentation. Then
\[G \mbox{ is of $\FF_q$-additive type} \longiff I_G = (\Prim_1(B_G)), \mbox{ locally on } S,\]
in analogy to Theorem \ref{embeddingsintoGa}.
\end{rem}

\begin{rem}
Let $G,H \in \Grpa qA$. If we have $G\into \Ga^N$ and $H\into \Ga^L$ in $\Gra qA$, then
\[G\times H \longinto \Ga^{N \cup L},\]
cf. Remark \ref{categorical}. Therefore, $G\times H \in \Grpa qA$. Conversely, the embeddings $G,H \into G \times H$ respect the $\FF_q$-actions. Hence if $G\times H$ is of $\FF_q$-additive type, then so are $G$ and $H$.

This does not generalize to arbitrary extensions, as the following example illustrates. That is, $\Grpa qA$ is not a Serre subcategory of $\Gra qA$ (indeed, nor is $\grpa qA \subseteq \gra qA$).
\end{rem}

\begin{ex}\label{cokerex}
Let $q\neq p$, and consider the following short exact sequence
\[0 \longto \alpha_p \longto \alpha_q \longto H \longto 0\]
in $\gra qA$. Note that $\alpha_p,\alpha_q \in \grpa qA$. Applying $\Hom_{\Gra qA}(-,\Ga)$, we get
\[0 \longto \Prim_1(B_H) \longto \Prim_1(B_{\alpha_q}) \isoto \Prim_1(B_{\alpha_p}).\]
Hence $I_H \neq (\Prim_1(B_H)) = 0$, and $H$ is not of $\FF_q$-additive type, by Remark \ref{addtypeiffprimgen}.
\end{ex}

\begin{thm}\label{frobsurj}
Let $G\in\gra qA$. Then $G$ is of $\FF_q$-additive type if and only if it is of additive type and the $p$-Frobenii
\[f_t\!: \Prim_{p^t}(B_G) \longto \Prim_{p^{t+1}}(B_G),\ x\longmapsto x^p,\]
are surjective, $0 \leq t < r-1$. Moreover, if $G\in\grpa qA$, the primitive elements decompose as
\[\Prim(B_G) = \bigoplus_{s=0}^{r-1} \Prim_{p^s}(B_G)\]
into eigenspaces. Equivalently, $\Prim_j(B_G) = 0$ for all $j \neq p^s$, $s\in\NN$.
\end{thm}
\begin{proof}
We may assume $A$ to be local, and let $\iota\!: G \into \Ga^N$ be an embedding in $\Gra qA$. Proposition \ref{addpoly} implies that the additive polynomials decompose as desired,
\[\P := \Prim(A[x_1,\ldots,x_N]) = \bigoplus_{s=0}^{r-1} \P_s,\]
with $\P_s = \Prim_{p^s}(A[x_1,\ldots,x_N])$. Let $k$ be the residue field of $A$. We have the epimorphism
\[\iota^*|_{\P} \otimes_A k\!: \P \otimes_A k = \Prim(k[x_1,\ldots,x_N]) \longonto \Prim(B_G \otimes_A k)\]
by Theorem \ref{mainthmDG}. Thus $\P \otimes_A k \onto \Prim(B_G) \otimes_A k$ is surjective, as well, by Remark \ref{primbasechange}. Now consider the filtration
\[\Prim(B_G) =: M \supseteq M^p \supseteq M^{p^2} \supseteq \ldots,\]
where $M^{p^t} := \im((\sigma_p^t)^*M \to M,\ x \mapsto x^{p^t})$. Then we may conclude that $\iota$ induces surjections
\[\iota^*\!: \P^{p^t}/\P^{p^{t+1}} \longonto M^{p^t}/M^{p^{t+1}}\]
for all $t \geq 0$, by Nakayama's lemma. Therefore,
\begin{equation}\label{primsurj}
\iota^*(\Prim A[x_1,\ldots,x_N]) = \Prim(B_G).
\end{equation}
But $\iota^*$ respects the $\FF_q$-action, so in fact $\iota^*(\P_s) = \Prim_{p^s}(B_G)$ for all $0 \leq s < r$. This settles the second part, and moreover implies that under the epimorphism $\iota^*$, we have
\[0 = \P_{t+1}/(\P_t)^p \longonto \Prim_{p^{t+1}}(B_G)/f_t(\Prim_{p^t}(B_G)).\]
Conversely, if all the $f_t$ are surjective, we have $I_G = (\Prim(B_G)) = (\Prim_1(B_G))$, and $G$ is of $\FF_q$-additive type by Remark \ref{addtypeiffprimgen}.
\end{proof}

\begin{rem}
In the previous proof, we do not need to invoke Theorem \ref{mainthmDG}. With the above notation, we may assume $k$ algebraically closed. Then $\Ext^1_{\Gr k}(\coker\iota \otimes k,\GG_{a,k})=0$, cf. \cite{Lau}, Lemma B.3.15. This implies as before the surjectivity of
\[\iota^*|_{\P}\!: \Hom(\Ga^N,\Ga) \longto \Hom(G,\Ga).\]
\end{rem}

\begin{lemma}\label{balequiv}
Let $G \in \grpa qA$, and let $f_t$ be the $p$-Frobenii from Theorem \ref{frobsurj}. The following conditions are equivalent.
\begin{enumerate}[leftmargin=10ex, label=\normalfont $(\roman*)$]
\itemsep3pt
\item The maps $f_t$ are bijective, for all $0 \leq t < r-1$.
\item The map $f'\!: \Prim_1(B_G) \to \Prim_{p^{r-1}}(B_G),\ x \mapsto x^{p^{r-1}}$, is injective.
\item The rank of $\Prim_{p^s}(B_G)$ is the same for all $0 \leq s \leq r-1$.
\item $\ord(G) = q^{\rk \Prim_1(B_G)}$.
\end{enumerate}
\end{lemma}
\begin{proof}
The first two conditions are equivalent by Theorem \ref{frobsurj}. Furthermore, by Nakayama, they are equivalent to $(iii)$. Now, Proposition \ref{orderofG} tells us that $\rk \Prim_{p^s}(B_G) = \rk \Prim_{1}(B_G)$ for all $0 \leq s \leq r-1$ if and only if
\[\ord(G) = p^{\rk(\Prim B_G)} = q^{\rk(\Prim_1 B_G)}.\]
Hence $(iii)$ is equivalent to $(iv)$.
\end{proof}

\begin{defn}
We say that the group scheme with $\FF_q$-action $G\in\grpa qA$ is balanced, if the equivalent conditions in Lemma \ref{balequiv} hold for $G$. The full subcategory of $\grpa qA$ of balanced group schemes will be called $\grb qA$.
\end{defn}

\begin{rem}
In \cite{Ta}, Taguchi defines $\grb qA$ using condition $(iv)$ from Lemma \ref{balequiv}. The first two conditions will be useful to generalize the definition to infinite group schemes.
\end{rem}

\begin{rem}
We have $\grb pA = \grp A$, since the condition on the $p$-Frobenii is empty.
\end{rem}

\begin{ex}\label{balancedex}
Consider $G = \alpha_{p^s} = \Spec A[x]/(x^{p^s})$ with the usual $\FF_q$-action $[\alpha]^*x = \alpha x$ for $\alpha\in\FF_q$. Then
\[\alpha_{p^s} \mbox{ is balanced} \longiff r|s.\]
Indeed, $\Prim(B_{\alpha_{p^s}}) = \span_A(x,x^p,\ldots,x^{p^{s-1}})$ by Proposition \ref{addpoly}, and therefore
\[\Prim_1(B_{\alpha_{p^s}}) = \span_A(x^{q^a} \mid 0 \leq a < s/r),\]
since $p^s \congr 1 \mod q-1 \longiff \frac{p^s-1}{q-1}\in\ZZ \longiff r|s$.
\end{ex}

\begin{rem}\label{balancedlocus}
Let $G \in \grpa qS$. If $S$ is connected, then
\[G \mbox{ is balanced} \longiff G_s \mbox{ is balanced for some point } s\in S,\]
because $\Prim(B_G)$ is locally free and stable under base change (Remark \ref{primbasechange}). Hence the balanced locus of $G\in\grpa qS$ in $S$ is a union of connected components. If $S$ is noetherian, it is thus closed and open.
\end{rem}

\begin{rem}\label{etimpliesb}
If $G \in \grpa qS$ is \'etale, its Frobenius is an isomorphism. Therefore, $G$ is balanced by Lemma \ref{balequiv}, $(i)$.
\end{rem}

\begin{lemma}\label{productb}
Let $G,H\in\grpa qA$. If two of $G$, $H$ and $G\times H$ are balanced, then so is the third.
\end{lemma}
\begin{proof}
Proposition \ref{primoftensor} implies that
\[\Prim_{p^s}(B_G \otimes B_H) = \Prim_{p^s}(B_G) \otimes 1 + 1 \otimes \Prim_{p^s}(B_H)\]
for all $0 \leq s \leq r-1$. Then it is clear that if two of $G$, $H$ and $G\times H$ satisfy condition $(i)$ from Lemma \ref{balequiv}, the third does as well.
\end{proof}

\begin{rem}
A posteriori, Lemma \ref{productb} holds for general extensions $0 \to G \to E \to H \to 0$. Namely, by Theorem \ref{mainthmFq}, we obtain a short exact sequence
\[0 \longto \Prim_1(B_H) \longto \Prim_1(B_E) \longto \Prim_1(B_G) \longto 0.\]
Then the statement follows using Lemma \ref{balequiv}, $(iv)$, together with Proposition \ref{orderofG}.
\end{rem}

\section{The functors \texorpdfstring{$\G$}{G} and \texorpdfstring{$\M$}{M}}\label{sec6}

We continue to denote by $A$ an $\FF_q$-algebra, for $q=p^r$.

\begin{defn}
The Dieudonn\'e $\FF_q$-functor is the contravariant functor
\[\M_q = \M\!: \grb qA \longto \sht qA,\ G \longmapsto (\Prim_1(B_G), x \mapsto x^q).\]
Recall that $\Prim_1(B_G) \isom \Hom_{\gra qA}(G,\Ga)$, and that it is locally free over $A$ (Remark \ref{eigendecomp}).
\end{defn}

\begin{rem}\label{Mwelldef}
$\M$ is well-defined, since $\Delta(x^q) = (x\otimes 1 + 1 \otimes x)^q = x^q \otimes 1 + 1 \otimes x^q$, and
\[[\alpha]^*x^q = \alpha^q x^q = \alpha x^q \mbox{ for } x\in \M(G),\ \alpha\in\FF_q.\]
\end{rem}

\begin{defn}
The Drinfel'd $\FF_q$-functor is defined to be the contravariant functor
\[\G_q = \G\!: \sht qA \longto \grb qA,\ (M,f) \longmapsto \Spec(\Sym(M)/\f),\]
where $\f$ is the ideal $\f = (x^{\otimes q} - f(x) \mid x \in M)$. Comultiplication and $\FF_q$-action are given by
\[\Delta(x) = x \otimes 1 + 1 \otimes x, \mbox{ and } [\alpha]^*x = \alpha x \mbox{ for } x\in M,\ \alpha\in\FF_q,\]
extended to the whole algebra.
\end{defn}

\begin{rem}\label{imageofG}
Let $(M,f)\in\sht qA$. The $\FF_q$-action on $\G(M,f)$ is well-defined, since
\[[\alpha]^*x^{\otimes q} = \alpha^q x^{\otimes q} = \alpha f(x) = [\alpha]^*f(x) \mbox{ for } \alpha\in\FF_q,\ x\in M.\]
Furthermore, locally on $\Spec A$, we can take a basis $x_1,\ldots,x_N$ of $M$ and the projection
\[\Sym(M) \longonto \Sym(M)/\f\]
will define an embedding $\G(M,f) \into \Ga^N$, which respects the $\FF_q$-actions. Moreover, note that the products $\{\prod_{i=1}^N x_i^{e_i} \mid 0 \leq e_i < q\}$ form a basis of $B_{\G(M,f)} = \Sym(M)/\f$. Therefore 
\begin{equation}\label{orderofGMf}
\ord(\G(M,f)) = q^{\rk M},
\end{equation}
and indeed $\G(M,f) \in \grb qA$ by Lemma \ref{balequiv}, $(iv)$.
\end{rem}

\begin{thm}\label{mainthmFq}
The Drinfel'd $\FF_q$-functor $\G\!: \sht qA \to \grb qA$ defines an exact anti-equivalence of categories with quasi-inverse $\M$.
\end{thm}
\begin{proof}
The proof is the same as for Theorem \ref{mainthm}, where we use \eqref{orderofGMf} in place of \eqref{orderofGM}. This yields that the adjunction morphisms $v_{(M,f)}$ are bijective. On the other hand,
\[u_G^*\!: \Sym(\Prim_1(B_G))/(x^{\otimes q} - x^q \mid x\in\Prim_1(B_G)) \longto B_G\]
is surjective, since $B_G$ is generated as an algebra by $\Prim_1(B_G)$, cf. Remark \ref{addtypeiffprimgen}. Finally,
\[\ord \G(\M(G)) = q^{\rk(\Prim_1 B_G)} = p^{\rk(\Prim B_G)} = \ord G,\]
by Proposition \ref{orderofG}, and because $G$ is balanced.
\end{proof}

Let us note some properties of the functor $\G$ (cf. \cite{Drin}, Proposition 2.1).

\begin{prop}\label{properties}
Let $(M,f)\in\sht qA$, and $I = I_{\G(M,f)}$ the augmentation ideal of $\G(M,f)$. The cotangent space of $\G(M,f)$ is described by
\[I/I^2 \isom \coker(f^{(q)})\]
as an $A$-module, where $f^{(q)}\!: \sigma_q^*M \to M$, as before. Moreover,
\begin{enumerate}[leftmargin=10ex, label=\normalfont $(\alph*)$]
\itemsep3pt
\item The group scheme $\G(M,f)$ is \'etale $\longiff f$ is bijective.
\item The fibres of $\G(M,f)$ are connected $\longiff f$ is nilpotent, locally on $\Spec A$.
\end{enumerate}
\end{prop}
\begin{proof}
For the first part, we see that the composition $\tau\!: M \into I \onto I/I^2$ is surjective, since every element of $I/I^2$ is represented by a linear polynomial in $M$. But
\[\ker(\tau) = M \cap I^2 = f(M).\]
Indeed, an element of $M$ lies in $I^2$ if and only if it is of the form
\[f(x) \congr x^{\otimes q} \in I^2 \mbox{ for some } x \in M.\]
The statements $(a)$,$(b)$ follow from the fact that $f$ is a power of the Frobenius on $\Sym(M)/\f$. Hence $f$ is bijective $\longiff \Frob_{\G(M,f)}$ is an isomorphism $\longiff \G(M,f)$ is \'etale (\cite{DG}, IV, $\mathsection$3, 5.3), because $\G(M,f)$ is finite flat and finitely presented.

Analogously, $f$ is locally nilpotent if and only if $\Frob_{\G(M,f)}$ is. But each fibre of $\G(M,f)$ is connected if and only if its Frobenius is nilpotent (\textit{loc.cit.}).
\end{proof}

\begin{rem}
If $A$ is a field, Theorem \ref{mainthmFq} in particular says that $\grb qA$ is an abelian category. On the other hand, it follows from Example \ref{cokerex} that for $q\neq p$, the category $\grpa qA$ is not abelian. The problem is of course that $\alpha_p$ is not balanced, by Example \ref{balancedex}.
\end{rem}

\begin{ex}\label{kerex}
Let $q \neq p$. Take $B_G = A[x_1,\ldots,x_r]/(x_1^p,\ldots,x_r^p)$ with $x_i \in \Prim_{p^{i-1}}(B_G)$. Clearly, $\rk \Prim_{p^s}(B_G)$ is the same for all $0 \leq s < r$, and $\Prim_{j} B_G = 0$ for all $j \neq p^s$. Then the Hopf algebra morphism
\[u^*\!: A[x]/(x^q) \longto B_G,\ x \longmapsto x_1\]
is compatible with the $\FF_q$-actions, i.e., it induces $u\!: G \to \alpha_q$ in $\gra qA$. Setting $I = (x)$ to be the augmentation ideal of $\alpha_q$, we see that
\[\ker(u) = \Spec A[x_2,\ldots,x_r]/(x_2^p,\ldots,x_r^p) = \Spec(B_G/u^*(I)B_G)\]
is not balanced. Of course, $G$ is not of $\FF_q$-additive type.
\end{ex}

In Theorem \ref{structurethmFq}, we will be able to describe the structure of our category over a perfect field $k$. This allows the following fibrewise characterization.

\begin{cor}
Let $G\in\grpa qS$ be fibrewise connected. Then $G$ is balanced if and only if it is of the form $\prod \alpha_{q^{s_i}}$ on all geometric fibres over $S$.
\end{cor}
\begin{proof}
We can check $G$ to be balanced on the fibres, by Remark \ref{balancedlocus} and since the condition is stable under base change (cf. Remark \ref{primbasechange}). By Theorem \ref{structurethmFq}, a geometric fibre of $G$ is balanced if and only if it is of the form $\prod \alpha_{q^{s_i}}$.
\end{proof}

The following result gives another perspective on the balance property (Remark \ref{altb}).

\begin{prop}
Assume that $A$ is an $\FF_{q^n}$-algebra, $n\geq 1$. Let $\F\!: \grpa{q^n}{A} \to \grpa qA$ be the forgetful functor. Then the following diagram commutes,
\begin{equation}\label{powerofq}
  \begin{tikzcd}
    \mathllap{(M,f)} \in \ar[d, mapsto, shift right=9ex] \sht{q^n}{A} \ar[r, "{\G_{q^n}}"] \ar[d] & \grpa{q^n}{A} \ar[d, "{\F}"] \\
    \mathllap{\Big(\bigoplus\limits_{i=0}^{n-1} (\sigma_q^i)^*M, F\Big)} \in \sht qA \ar[r, "{\G_q}"] & \grpa qA,
  \end{tikzcd}
\end{equation}
where $F\!: (x_0,\ldots,x_{n-1}) \mapsto (x_1,\ldots,x_{n-1},f(x_0))$ is the matrix $\smash[t]{\begin{pmatrix} 0 & 1 & & 0 \\ & \ddots & \ddots & \\ & & \ddots & 1 \\ f & & & 0 \end{pmatrix}}$.
\end{prop}
\begin{proof}
We denote the scalar multiplication on $(\sigma_q^i)^*M$ by $\lambda.x = \lambda^{q^i}x$ for $\lambda\in A,\ x\in M$ (the usual action without the dot). Since $f$ is $q^n$-linear, we get
\[F(\lambda.(x_0,\ldots,x_{n-1})) = (\lambda^q x_1,\ldots,\lambda^{q^{n-1}}x_{n-1},\lambda^{q^n}f(x_0)) = \lambda^q.F(x_0,\ldots,x_{n-1}),\]
for all $(x_0,\ldots,x_{n-1}) \in M' := \bigoplus\limits_{i=0}^{n-1} (\sigma_q^i)^*M$. We have to show that
\[\Sym(M)/\f_M \isom \Sym(M')/\f_{M'} \mbox{ in } \hopfa qA,\]
where $\f_M = (x^{\otimes q^n} - f(x) \mid x \in M)$ and $\f_{M'} = (\underline x^{\otimes q} - F(\underline x) \mid \underline x \in M')$. We define
\[\phi\!: \Sym(M) \longto \Sym(M')/\f_{M'} \mbox{ via } M \ni x \longmapsto (x,0,\ldots,0) \in M',\]
extended to an algebra morphism. Note that for $x \in M$, we have
\[\phi(x^{q^i}) = (x,0,\ldots,0)^{q^i} = F^i(x,0,\ldots,0) = (\ldots,0,x,0,\ldots) \in (\sigma_q^i)^*M \subseteq M'.\]
Now, $\phi$ factors through the quotient, because
\[\phi(x^{q^n} - f(x)) = F^n(x,0,\ldots,0) - (f(x),0,\ldots,0) = 0.\]
Finally, $\bar\phi\!: \Sym(M)/\f_M \to \Sym(M')/\f_{M'}$ is an isomorphism in $\Hopfa qA$, since by definition,
\[\bar\phi(\Prim_1(\Sym(M)/\f_M)) = \Prim_1(\Sym(M')/\f_{M'}),\]
and locally on $\Spec(A)$, it maps bases to bases.
\end{proof}

\begin{rem}\label{altb}
Consider the following diagram,
\begin{equation}\label{powerofp}
  \begin{tikzcd}
    \grpa qA \ar[r, "{\M_q}"] \ar[d, "{\F}"'] & \sht qA \ni \mathrlap{(M,f)} \ar[d] \ar[d, mapsto, shift left=9ex] \\
    \grpa pA \ar[r, "{\M_p}"] & \sht pA \ni \mathrlap{\Big(\bigoplus\limits_{t=0}^{r-1} (\sigma_p^t)^*M, F\Big)}
  \end{tikzcd}
\end{equation}
corresponding to the commutative diagram \eqref{powerofq}. Requiring \eqref{powerofp} to commute recovers the balance condition. Namely, for $G\in\grpa qA$, we have
\[\Prim(B_G) \isom \bigoplus\limits_{t=0}^{r-1} (\sigma_p^t)^*\Prim_1(B_G)\] if and only if all the (linearized) maps
\[(\sigma_p^t)^*\Prim_1(B_G) \longto \Prim_{p^t}(B_G),\ x \longmapsto x^{p^t},\ (0 \leq t < r)\]
are isomorphisms. This is equivalent to condition $(ii)$ in Lemma \ref{balequiv}.
\end{rem}

\begin{rem}
The functor $\Sigma\!: \sht{q^n}{A} \to \sht qA,\ (M,f) \mapsto \Big(\bigoplus\limits_{i=0}^{n-1} (\sigma_q^i)^*M, F\Big)$ from the above diagram \eqref{powerofq} has a left adjoint. Namely, there are bifunctorial isomorphisms
\begin{equation}\label{powerofqadj}
\Hom_{\sht{q^n}{A}}((M,f^n),(M',f')) \isoto \Hom_{\sht qA}((M,f),\Sigma(M',f')),
\end{equation}
for all $(M,f)\in\sht qA$ and $(M',f')\in\sht{q^n}{A}$. Indeed, for $\Phi\!: (M,f^n) \to (M',f')$, set
\[\Phi'\!: (M,f) \longto \Sigma(M',f'),\ x \longmapsto (\Phi(x),\Phi(f(x)),\ldots,\Phi(f^{n-1}(x))) \in \Big(\bigoplus\limits_{i=0}^{n-1} (\sigma_q^i)^*M', F'\Big).\]
Then $\Phi'$ is a morphism of shtukas, because $f'(\Phi(x)) = \Phi(f^n(x))$ for $x\in M$, and thus
\[F'(\Phi'(x)) = (\Phi(f(x)),\ldots,\Phi(f^{n-1}(x)),f'(\Phi(x))) = (\Phi(f(x)),\ldots,\Phi(f^n(x))) = \Phi'(f(x)).\]
Conversely, if $\Psi\!: (M,f) \to \Sigma(M',f')$ lies in the right hand side of \eqref{powerofqadj}, we obtain a morphism on the left by composition with the projection $\pr_0\!: \bigoplus\limits_{i=0}^{n-1} (\sigma_q^i)^*M' \to M',\ (x_i) \mapsto x_0$, say
\[\Psi' := \pr_0 \circ \Psi\!: (M,f^n) \longto (M',f'),\ x \longmapsto \Psi(x)_0.\]
Applying $\pr_0$ to $\Psi(f^n(x)) = (F')^n(\Psi(x)) = (F')^{n-1}(\ldots,f'(\Psi'(x)) = (f'(\Psi'(x)),\ldots)$ yields
\[\Psi'(f^n(x)) = f'(\Psi'(x)).\]
By Theorem \ref{mainthmFq} and commutativity of \eqref{powerofq}, we can transfer \eqref{powerofqadj} to an adjunction
\[\Hom_{\grb{q}{A}}(\F(G),H) \isoto \Hom_{\grb{q^n}{A}}(G,\Omega(H)).\]
That is, $\Omega\!: \grb{q}{A} \to \grb{q^n}{A}$ is a right adjoint to $\F\!: \grb{q^n}{A} \to \grb qA$, the forgetful functor. Unfortunately, it is not clear how to describe $\Omega$ intrinsically. One might guess that it is given by Serre's tensor construction \cite{Con}, Theorem 7.2. This induces a functor
\[\grpa{q}{A} \longto \grpa{q^n}{A},\ G \longmapsto G \otimes_{\FF_q} \FF_{q^n},\]
where $(G \otimes_{\FF_q} \FF_{q^n})(R) = G(R) \otimes_{\FF_q} \FF_{q^n}$ on points. The $\FF_{q^n}$-operation is on the right variable. Then $\Ga \otimes_{\FF_q} \FF_{q^n} \isom \Ga^n$, equivariantly, where $\Ga^n$ carries its usual (product) $\FF_{q^n}$-action. This indeed agrees with "$\Omega(\Ga)$" $\!= \G_{q^n}(\N(\M_q(\Ga)))$, where $\N(M,f) = (M,f^n)$ is the left adjoint of $\Sigma$ from above. Moreover, $\Fq \otimes_{\FF_q} \FF_{q^n} \isom \underline{\FF_{q^n}} = \Omega(\Fq)$. However, if $n>1$, already as schemes,
\[\alpha_q \otimes_{\FF_q} \FF_{q^n} \isom \alpha_q^n \not\isom \alpha_{q^n} = \G_{q^n}(\N(\M_q(\alpha_q))) = \Omega(\alpha_q).\]
More precisely, since Serre's tensor construction is exact, $\alpha_q \otimes_{\FF_q} \FF_{q^n} \into \Ga  \otimes_{\FF_q} \FF_{q^n}$, and this is the product embedding $\alpha_q^n \into \Ga^n$. On the other hand, the composition $\Omega = \G_{q^n} \circ \N \circ \M_q$ maps $\alpha_q \into \Ga$ to the embedding into the first coordinate $\alpha_{q^n} \into \Ga^n$.
\end{rem}

\section{Quasi-balanced group schemes}\label{sec7}

Let $A$ be an $\FF_q$-algebra, $q=p^r$. For $G\in\grpa qA$, consider the eigenspace decomposition
\[I_G = \bigoplus_{j=1}^{q-1} I_j\]
for the $\FF_q^{\times}$-action on the augmentation ideal of $G$, cf. \eqref{decomp}.

\begin{defn}
A group scheme $G\in \grpa qA$ is called quasi-balanced if $\rk(I_j)$ is the same for all $1 \leq j \leq q-1$.
\end{defn}

\begin{rem}\label{qpower}
Let $G\in \grpa qA$ be quasi-balanced. By Proposition \ref{orderofG}, $\rk(I_G) = p^N - 1$, where $N = \rk(\Prim(B_G))$. Then we have $\frac{p^N-1}{q-1} \in \ZZ,$ and thus $r | N$, say $rn = N$. This yields
\[\rk(I_j) = q^{n-1} + \ldots + q + 1 \mbox{ for all } 1 \leq j \leq q-1,\] and of course $\ord G = q^n$. Note that the analogue of Remark \ref{balancedlocus} holds, i.e. for $G\in\grpa qS$, the quasi-balanced locus of $G$ in $S$ is closed and open (if the base is noetherian).
\end{rem}

\begin{lemma}\label{bimpliesqb}
Every $G\in\grb qA$ is quasi-balanced.
\end{lemma}
\begin{proof}
We may assume that $A$ is a local ring. If $x_1,\ldots,x_n$ is a basis of $\Prim_1(B_G)$, then
\[\{\prod_{i=1}^n x_i^{e_i} \mid 0 \leq e_i < q,\ (e_1,\ldots,e_n) \neq 0\}\]
is a basis of $I_G$ (cf. Remark \ref{imageofG}), since $G$ is balanced. On this basis, $\alpha\in\FF_q$ acts via
\[[\alpha]^*\prod_{i=1}^n x_i^{e_i} = \alpha^{\sum e_i}\prod_{i=1}^n x_i^{e_i}.\]
Therefore, it decomposes into eigenbases for $I_j$,
\[\{\prod_{i=1}^n x_i^{e_i} \mid 0 \leq e_i < q,\ (e_1,\ldots,e_n) \neq 0,\ \sum_{i=1}^n e_i \congr j \mod q-1\}.\]
In order to count the ranks, we identify the bases with
\[E_j^{(n)} := \{0 \neq \underline e = (e_1,\ldots,e_n) \mid 0 \leq e_i < q,\ \sum_{i=1}^n e_i \congr j \mod q-1\}.\]
We claim that
\[\# E_j^{(n)} = q^{n-1} + \ldots + q + 1 \mbox{ for all } 1 \leq j \leq q-1.\]
Let us prove this by induction on $n$. For $n=1$, there is nothing to show. For $n > 1$, we have
\[E_j^{(n)} = \coprod_{e=0}^{q-1} \{\underline e \in E_j^{(n)} \mid e_n = e\} = \coprod_{e=0}^{q-1} \{(\underline e, e) \mid \underline e \in E_{j-e}^{(n-1)}\} \amalg \{(0,\ldots,0,j)\}.\]
Therefore indeed: $\# E_j^{(n)} = q(q^{n-2} + \ldots + q + 1) + 1 = q^{n-1} + \ldots + q + 1$.
\end{proof}

\begin{rem}
For $G\in\grp A$, i.e. $q=p$, the condition of being quasi-balanced is therefore automatic. We can also see this concretely via $p$-adic expansion. Namely, let
\[\rho\!: \ZZ/(p^n-1) \longonto \ZZ/(p-1)\]
be the projection, where $\ord G = p^n$, as above. Then we have the bijection
\[E_j = E_j^{(n)} \isoto \rho^{-1}(j),\ (e_1,\ldots,e_n) \longmapsto \sum_{i=1}^n e_ip^i,\]
and thus again $\# E_j = \frac{p^n-1}{p-1} = p^{n-1} + \ldots + p + 1$.
\end{rem}

\begin{lemma}
Let $G,H \in \grpa qA$. For $1 \leq j \leq q-1$, the eigenspaces in the product satisfy
\begin{equation}\label{productformula}
\rk I_j(G\times H) = \!\!\sum_{\substack{k+l \congr j \mod q-1 \\ 0 \leq k,l \leq q-1}}\!\! \rk I_k(G) \cdot \rk I_l(H),
\end{equation}
with the convention $I_0(-) := A$.
\end{lemma}
\begin{proof}
The product decomposes as follows,
\[B_G \otimes_A B_H = \Big(\bigoplus_{k=0}^{q-1} I_k(G)\Big) \otimes_A \Big(\bigoplus_{l=0}^{q-1} I_l(H)\Big) = \bigoplus_{k,l} \Big(I_k(G) \otimes_A I_l(H)\Big),\]
and, of course, whenever $k+l \congr j \mod q-1$, we have
\[I_k(G) \otimes_A I_l(H) \subseteq I_j(G \times H),\]
by definition of the product $\FF_q$-action.
\end{proof}

\begin{cor}\label{productqb}
Let $G,H\in\gra qA$. If two of $G$, $H$ and $G\times H$ are quasi-balanced, then so is the third.
\end{cor}
\begin{proof}
Let $G,H$ be quasi-balanced. By Remark \ref{qpower}, we have $\ord G = q^n$ and $\ord H = q^m$ for some $n,m\in\NN$. Then for any $1 \leq j \leq q-1$, the product formula \eqref{productformula} becomes
\begin{equation*}
  \begin{aligned}
    \rk I_j(G \times H) & = \!\sum_{\substack{k+l \congr j \mod q-1 \\ 1 \leq k,l \leq q-1}}\! \rk I_k(G) \cdot \rk I_{l}(H) + \rk I_j(G) + \rk I_j(H) \\
    & = (q-1)(q^{n-1} + \ldots + q + 1)(q^{m-1} + \ldots + q + 1) + \rk I_j(G) + \rk I_j(H) \\
    & = (q^n-1)(q^{m-1} + \ldots + q + 1) + (q^{n-1} + \ldots + q + 1) + (q^{m-1} + \ldots + q + 1) \\
    & = q^n(q^{m-1} + \ldots + q + 1) + q^{n-1} + \ldots + q + 1 \\
    & = q^{n+m-1} + \ldots + q^n + q^{n-1} + \ldots + q + 1.
  \end{aligned}
\end{equation*}
We conclude that $G\times H$ is quasi-balanced.

Conversely, assume that $H$ and $G\times H$ are quasi-balanced. Since by Remark \ref{qpower},
\[\ord H = q^m, \mbox{ and } \ord(G\times H) = q^{n+m},\]
for some $m,n\in\NN$, we have $\ord G = q^n$. Applying \eqref{productformula} again, we get for $1 \leq j \leq q-1$ that
\begin{equation*}
  \begin{aligned}
    q^{n+m-1} + \ldots + q + 1 & = \sum_{1 \leq k \leq q-1} \rk I_k(G)(q^{m-1} + \ldots + q + 1) + \rk I_j(G) + \rk I_j(H) \\
    & = (q^n-1)(q^{m-1} + \ldots + q + 1) + \rk I_j(G) + (q^{m-1} + \ldots + q + 1) \\
    & = q^{n+m-1} + \ldots + q^n + \rk I_j(G),
  \end{aligned}
\end{equation*}
hence the claim.
\end{proof}

\begin{rem}\label{combinatorics}
Let us consider group schemes of the form
\[G = \Spec(A[x_1,\ldots,x_h]/(x_1^{p^{s_1}},\ldots,x_h^{p^{s_h}})) = \prod_{i=1}^h \alpha_{p^{s_i}},\]
so that all $x_i \in \Prim_1(B_G)$. Consider the standard basis of $I_G$,
\[\{\prod_{i=1}^h x_i^{e_i} \mid 0 \leq e_i < p^{s_i},\ (e_1,\ldots,e_h) \neq 0\}.\]
As always, it decomposes into eigenbases for the $I_j$. This yields
\[\rk I_j = \!\sum\limits_{\substack{a \congr j \mod q-1 \\ a \neq 0}}\! n_a,\]
where
\[n_a := \#\{0 \neq (e_1,\ldots,e_h) \mid 0 \leq e_i < p^{s_i},\ \sum_{i=1}^h e_i = a\}.\]
Note that $n_a$ is precisely given by the coefficient of $X^{a}$ in the polynomial
\begin{equation}\label{S}
S(X) = (X^{p^{s_1}-1} + \ldots + X + 1)\cdot\ldots\cdot (X^{p^{s_h}-1} + \ldots + X + 1) \in \ZZ[X].
\end{equation}
\end{rem}

\begin{ex}\label{theexample}
Let $q = 4$, and consider the special case
\[G = \Spec(A[x_1,\ldots,x_6]/(x_1^p,\ldots,x_6^p)).\]
Then $G\in\grpa qA$ is quasi-balanced, since
\[S(X) = (X+1)^6 = X^6 + 6X^5 + 15X^4 + 20X^3 + 15X^2 + 6X + 1.\]
Hence $\rk I_j = 21$ for all $1 \leq j \leq 3$. But obviously $G\notin\grb qA$, because $\Prim_p(B_G) = 0$. Thus the converse to Lemma \ref{bimpliesqb} is false in general.
\end{ex}

However, this is essentially the ``only'' counter-example, as the following results show.

\begin{prop}\label{lisa}
Let $G = \alpha_{p^{s_1}} \times \ldots \times \alpha_{p^{s_h}}$ as in Remark \ref{combinatorics}.
\begin{enumerate}
\itemsep3pt
\item If $q \neq 4$, then $G$ is quasi-balanced if and only if $r | s_i$ for all $1 \leq i \leq h$.
\item If $q=4$, then $G$ is quasi-balanced if and only if $6 | h' := \#\{i \mid s_i \not\congr 0 \mod r\}$.
\end{enumerate}
\end{prop}
\begin{proof}[Proof (joint with Sauermann)]
By Corollary \ref{productqb}, $G$ is quasi-balanced if and only if $G \times \alpha_{q^m}$ is quasi-balanced, for $m\in\NN$. Eliminating the factors of the form $\alpha_{q^m}$ from $G$, we can therefore assume that we have $s_i \not\congr 0 \mod r$ for all $i$. Then we have to show that
\[G \mbox{ is quasi-balanced if and only if } h = 0 \mbox{ in case } (1), \mbox{ resp. } 6|h \mbox{ in case } (2).\]
Note that for $r=1$, the claims are vacuous.

First, let us assume that $G$ is quasi-balanced. We keep the notation from Remark \ref{combinatorics}. Then evaluating \eqref{S} at the primitive root of unity $\zeta = e^{\frac{2\pi i}{q-1}} \in\mu_{q-1}(\CC)$ yields
\[S(\zeta) = \sum_{a \geq 0} n_a\zeta^a = \sum_{j=0}^{q-1} \rk(I_j) \zeta^j = 1 + \rk(I_1)\sum_{j=1}^{q-1} \zeta^j = 1 + \rk(I_1)\frac{\zeta^{q-1}-1}{\zeta - 1} = 1.\]
Let $0 < t_i < r$ with $s_i \congr t_i \mod r$. Then
\begin{equation}\label{Sofzeta}
1 = S(\zeta) = \prod_{i=1}^h (\zeta^{p^{s_i}-1} + \ldots + \zeta + 1) = \prod_{i=1}^h\frac{\zeta^{p^{s_i}}-1}{\zeta - 1} = \prod_{i=1}^h\frac{\zeta^{p^{t_i}}-1}{\zeta - 1}.
\end{equation}
On the other hand, $|\zeta^j - 1| \geq |\zeta - 1|$ for all $1 \leq j < q-1$. Hence by \eqref{Sofzeta} we must have
\[|\zeta^{p^{t_i}} - 1| = |\zeta - 1| \mbox{ for all } 1 \leq i \leq h.\]
Since $0 < t_i < r$, this implies that $p^{t_i} \congr -1 \mod q-1$, and therefore
\[q = p^{t_i} + 2, \mbox{ for } 1 \leq i \leq h.\]
Thus $q=4$ (unless $h=0$), and $t_i=1$ for all $i$. In that case then, equation \eqref{Sofzeta} reads
\[(\zeta + 1)^h = 1.\]
But $\zeta + 1 = e^{\frac{2\pi i}{3}} + 1 = e^{\frac{2\pi i}{6}}$, and therefore $6 | h$.

Conversely, the ``if''-direction in (1) is trivial ($h=0$). In case (2), keeping in mind \eqref{Sofzeta},
\[\rk(I_1) + \rk(I_2)\zeta + \rk(I_3)\bar\zeta = S(\zeta) - 1 = (\zeta + 1)^h - 1 = 0.\]
Conjugating this equation, we see that indeed $\rk(I_1) = \rk(I_2) = \rk(I_3)$.
\end{proof}

\begin{thm}
Let $q \neq 4$. Then $G\in\grpa qA$ is quasi-balanced if and only if it is balanced.
\end{thm}
\begin{proof}
We have already seen ``$\Leftarrow$'' in Lemma \ref{bimpliesqb}. Now let $G$ be quasi-balanced. Since the condition is stable under base change, we can assume $A$ to be a perfect field. As in (the proof of) Theorem \ref{structurethmFq}, resp. Theorem \ref{structurethm}, we have an $\FF_q$-equivariant decomposition
\[G = \pi_0(G) \times H.\]
By Remark \ref{etimpliesb}, $\pi_0(G)$ is balanced, hence quasi-balanced, and so $H$ is quasi-balanced as well by Corollary \ref{productqb}. Again as in Theorem \ref{structurethmFq},
\[B_H \isom A[x_1,\ldots,x_h]/(x_1^{p^{s_1}},\ldots,x_h^{p^{s_h}}) \mbox{ in } \hopfa qA.\]
Therefore, by Proposition \ref{lisa}, since $H$ is quasi-balanced, $r|s_i$ for all $1 \leq i \leq h$. Thus $H$ is balanced, and hence so is $G$.
\end{proof}

\section{The infinite case}\label{sec8}

In this section, we detail some of the difficulties one encounters when trying to transfer the theory from the finite to the infinite case.

\begin{defn}
The category $\Sht qA$ consists of pairs $(M,f)$ of a flat $A$-module $M$ together with a $q$-linear endomorphism $f$ of $M$. Morphisms in $\Sht qA$ are defined as in $\sht qA$. Note that again $\Sht pA = \FMod A$.

The shtuka $(M,f)$ is called locally finitely generated if locally over $\Spec A$, there exist some $x_1,\ldots,x_N \in M$ such that
\[M = \{\sum_{i=1}^N \sum_{a=0}^d \lambda_a f^a(x_i) \mid \lambda_a \in A,\ d \in \NN\}.\]
We denote by $\Shtfg qA$ the full subcategory of $\Sht qA$ of locally finitely generated shtukas.
\end{defn}

Trying to define the functor $\M$ in general, the first problem to arise is the following.

\begin{conj}\label{flatconj}
Let $G\in\Grp A$. Then the $A$-module $\Prim(B_G)$ is flat.
\end{conj}

Let us provisionally include the condition in the definition, and denote the corresponding full subcategories by $\Grpf A$, $\Grpaf qA$, and so forth. Then we may define
\[\M_q = \M\!: \Grpaf qA \longto \Sht qA, \mbox{ and } \G_q = \G\!: \Sht qA \longto \Grpaf qA,\]
as before. Recall that $\M_q(G)$ is flat if $\Prim(B_G)$ is, cf. Remark \ref{eigendecomp}.

\begin{rem}
The routine verifications in $\mathsection$\ref{sec6}, together with the following, show that the two functors are well-defined. Let $(M,f)\in\Sht qA$. We can write $M = \colim M_i$, where the $M_i$ are finitely generated free $A$-modules, by \cite{Laz}, Th\'eor\`eme 1.2. For each $i$, choose a basis $N_i$ of $M_i$, and set $N := \coprod N_i$. Denote by $\gamma_i\!: M_i \into \Sym(M_i) \to \colim \Sym(M_i)$. Then
\[\begin{tikzcd}[row sep=0em, column sep=1.5em]
  A[x_n \mid n \in N] \ar[r, two heads] & \colim \Sym(M_i) \isom \Sym(M) \ar[r, two heads] & \Sym(M)/\f, \\
  x_n \ar[r, mapsto] & \gamma_i(n), \mbox{ when } n \in N_i, & 
\end{tikzcd}\]
yields $\G(M,f) \into \Ga^N$, by universality $\FF_q$-equivariantly. Thus $\G(M,f)$ is of $\FF_q$-additive type.
\end{rem}

The adjunction (cf. \cite{Lau}, Lemma B.3.9) holds in this generality.

\begin{lemma}\label{adjunction}
The functors $\G$ and $\M$ form an adjoint pair $(\G,\M)$ of $\FF_q$-linear functors, that is, there exist bifunctorial isomorphisms of $\FF_q$-vector spaces
\[\Hom_{\Gra qA}(G,\G(M,f)) \isoto \Hom_{\Sht qA}((M,f),\M(G)).\]
In particular, $\G$ is right-exact and $\M$ is left-exact.
\end{lemma}
\begin{proof}
We use $\Hom_{\Gra qA}(G,\G(M,f)) \isom \Hom_{\Hopfa qA}(\Sym(M)/\f, B_G)$. Now the map
\begin{equation}\label{adjbij}
\Hom_{\Hopfa qA}(\Sym(M)/\f, B_G) \longto \Hom_{\Sht qA}((M,f),\M(G)),\ \phi \longmapsto \phi|_M,
\end{equation}
is well-defined and bijective. Indeed, by definition, we have $M \subseteq \Prim_1(\Sym(M)/\f)$, and hence $\phi(M) \subseteq \Prim_1(B_G)$. Furthermore, $\phi|_M$ is a morphism of shtukas. Namely, the diagram
\[\begin{tikzcd}
  M \ar[r, "{\phi|_M}"] \ar[d, "f"'] & \Prim_1(B_G) \ar[d, "{z \mapsto z^q}"] \\
  M \ar[r, "{\phi|_M}"] & \Prim_1(B_G)
\end{tikzcd}\]
commutes, because $\phi(x^{\otimes q} - f(x)) = 0$ for all $x\in M$. The inverse map of \eqref{adjbij} is given by
\[\Hom_{\Sht qA}((M,f),\M(G)) \longto \Hom_{\Hopfa qA}(\Sym(M)/\f, B_G),\ \Phi \longmapsto \hat\Phi,\]
where $\hat\Phi$ is the extension of $\Phi$ to $\Sym(M)$, which descends to $\Sym(M)/\f$. Indeed, the diagram
\[\begin{tikzcd}
  M \ar[r, "{\Phi}"] \ar[d, "f"'] & \Prim_1(B_G) \ar[d, "{z \mapsto z^q}"] \\
  M \ar[r, "{\Phi}"] & \Prim_1(B_G)
\end{tikzcd}\]
commutes, and thus $\hat\Phi(x^{\otimes q} - f(x)) = \Phi(x)^q - \Phi(f(x)) = 0$ for $x\in M$. Finally, $\hat\Phi$ clearly respects the coalgebra structure and $\FF_q$-action, since it is universal.

Now, $\Hom_{\Gra qA}(-,\Ga)$ is additive, and it follows from Proposition \ref{primoftensor} that
\[\Prim_1(B_G \otimes B_H) = \Prim_1(B_G) \otimes 1 + 1 \otimes \Prim_1(B_H).\]
Hence $z \mapsto z^q$ in $\M(G\times H)$ is indeed given by $(x,y) \mapsto (x^q,y^q)$ in $\M(G) \oplus \M(H)$, and therefore $\M$ is additive. By adjunction, so is $\G$. Note that it is evident that $\G$ sends surjective morphisms to closed embeddings in $\Gra qA$.

The $\FF_q$-linearity is trivial, as it reduces to the statement that $[\alpha]^*$, for $\alpha\in\FF_q$, acts on the eigenspace $\Prim_1(-)$ by scalar multiplication.
\end{proof}

In order to show the analogue of Theorem \ref{mainthmFq}, we would like to restrict ourselves to group schemes which are locally of finite presentation over the base. However, the following question remains open.

\begin{conj}\label{exactconj}
Let $G \in\Grpf A$ be locally finitely presented over $A$. Then there exists a closed embedding $\iota\!: G \into \Ga^N$, $N\in\NN$, locally on $\Spec A$, such that the induced morphism
\[\iota^*\!: \M_p(\Ga^N) = A[F]^N \longto \M_p(G)\]
is surjective.
\end{conj}

\begin{rem}
It seems reasonable to presume that if $G$ satisfies Conjecture \ref{exactconj}, then indeed any embedding $G\into\Ga^N$, $N\in\NN$, induces a surjection on the primitive elements. Certainly, if $A=k$ is a field, this is true by exactness of $\M$, cf. Theorem \ref{mainthmDG}.

Moreover, recall that it holds for finite $G$ over any base $A$, cf. \eqref{primsurj}.
\end{rem}

\begin{rem}
The restricted functor $\M\!: \Grpaffp qA \!\to \Shtfg qA$ is well-defined, assuming Conjecture \ref{exactconj} holds. For this, we have to see that for a surjective morphism as there,
\[\iota^*\!: A[x_1,\ldots,x_N] \longonto B_G \mbox{ in } \Hopfa qA,\]
the elements $\iota^*(x_1),\ldots,\iota^*(x_N) \in \Prim_1(B_G)$ form a system of generators under $x \mapsto x^q$. But indeed, recall from Proposition \ref{addpoly} that
\[\Prim_1(A[x_1,\ldots,x_N]) = \span_A(x_i^{q^a} \mid 1\leq i \leq N,\ a \in \NN).\]
On the other hand, Conjecture \ref{exactconj} implies that $\M(\iota)=\iota^*|_{\P_1}$ surjects onto $\Prim_1(B_G)$, as in the proof of Theorem \ref{frobsurj}.
\end{rem}

\begin{lemma}\label{balequivINF}
Assume Conjecture \ref{exactconj}. For $G \in \Grpaffp qA$, the following are equivalent.
\begin{enumerate}[leftmargin=10ex, label=\normalfont $(\roman*)$]
\itemsep3pt
\item $f_t\!: \Prim_{p^t}(B_G) \to \Prim_{p^{t+1}}(B_G),\ x\mapsto x^p$, is bijective, for $0 \leq t < r-1$.
\item The map $f'\!: \Prim_1(B_G) \to \Prim_{p^{r-1}}(B_G),\ x \mapsto x^{p^{r-1}}$, is injective.
\end{enumerate}
\end{lemma}
\begin{proof}
The claim tautologically follows from the analogue of Theorem \ref{frobsurj}, the key ingredient in the proof of which is precisely Conjecture \ref{exactconj}.
\end{proof}

\begin{defn}
We say $G\in\Grpaffp qA$ is balanced if the conditions in Lemma \ref{balequivINF} hold. We denote by $\Grbffp qA$ the full subcategory of $\Grpaffp qA$ of balanced group schemes.
\end{defn}

\begin{rem}
The functor $\G\!: \Shtfg qA \to \Grbffp qA$ is well-defined. In order to see this, let $(M,f)\in\Shtfg qA$ and locally on $\Spec A$, choose a system of generators $x_1,\ldots,x_N$ of $(M,f)$. Then this yields, locally over $\Spec A$, an epimorphism
\[A[x_1,\ldots,x_N] \longonto \Sym(M)/\f.\]
Indeed, all elements of $\Sym(M)/\f$ are polynomial in the $x_i$, as we can write any $x\in M$ as
\[x = \sum_{i=1}^N\sum_{a=0}^d \lambda_af^a(x_i) = \sum_{i=1}^N\sum_{a=0}^d \lambda_ax_i^{\otimes q^a} \mbox{ in } \Sym(M)/\f.\]
Keeping in mind Theorem \ref{frobsurj}, it then moreover follows that for all $0 \leq s < r$, the maps
\[M=\Prim_1(\Sym(M)/\f) \longonto \Prim_{p^s}(\Sym(M)/\f),\ x\longmapsto x^{p^s},\]
are bijective. We conclude that $\G(M,f)$ is balanced, by Lemma \ref{balequivINF}, $(ii)$.
\end{rem}

\begin{rem}\label{primbcINF}
Let $G\in\Grpaffp qA$ and $\iota\!: G \into \Ga^N$. Let further $R$ be the residue field at a point $s\in\Spec A$. Then $\iota_s\!: G_s \into \GG_{a,s}^N$ is a closed embedding as well. Therefore,
\[\begin{tikzcd}
  \Prim(A[x_1,\ldots,x_N]) \otimes_A R \ar[r, two heads] \ar[d, "{(*)}"] & \Prim(B_G) \otimes_A R \ar[d, hook] \\
  \Prim(R[x_1,\ldots,x_N]) \ar[r, two heads] & \Prim(B_G \otimes_A R)
\end{tikzcd}\]
by Conjecture \ref{exactconj} (and recalling Remark \ref{primbasechange}). We can see for example from Proposition \ref{addpoly} that $(*)$ is an isomorphism. Thus, the primitive elements are stable under taking fibres.
\end{rem}

Note that if we assume Conjecture \ref{exactconj}, then we only need Conjecture \ref{flatconj} to hold for locally finitely presented $G$. Indeed, we can then prove the adjunction of $\G$ and $\M$ as functors between $\Grbfp qA$ and $\Shtfg qA$ by the identical argument.

\begin{thm}
Assume that Conjectures \ref{flatconj} and \ref{exactconj} are true. The functor
\[\G\!: \Shtfg qA \longto \Grbfp qA\]
defines an anti-equivalence of categories with quasi-inverse $\M$.
\end{thm}
\begin{proof}
We follow \cite{DG}, IV, $\mathsection$3, 6.5. First of all, we may assume that $A$ is a local ring, and in fact even an Artin local ring, by Remark \ref{primbcINF}. We consider the short exact sequence
\begin{equation}\label{uGsequ}
0 \longto G \xto{\ u_G\ } \G(\M(G)) \longto Q \longto 0,
\end{equation}
defined by the adjunction morphism. On the other hand, we once again have the isomorphism
\[v_{\M(G)}^{-1} = \M(u_G)\!: \M(\G(\M(G))) \isoto \M(G).\]
Since $G$ is balanced, $\M(u_G)$ extends to all primitive elements, i.e.
\[\Hom_{\Gr A}(u_G,\Ga)\!: \Hom_{\Gr A}(\G(\M(G)),\Ga) \isoto \Hom_{\Gr A}(G,\Ga).\]
Then applying $\M_p = \Hom_{\Gr A}(-,\Ga)$ to \eqref{uGsequ} yields the exact sequence
\begin{equation}\label{MofuGsequ}
0 \longto \Hom_{\Gr A}(Q,\Ga) \longto \Hom_{\Gr A}(\G(\M(G)),\Ga) \isoto \Hom_{\Gr A}(G,\Ga).
\end{equation}
The fibre of $Q$ over the closed point of $\Spec A$ is unipotent by \cite{DG}, IV, $\mathsection$2, 2.3, i.e. there is a non-trivial morphism to the additive group.

Since $A$ is an Artin ring, this lifts to $Q \to \Ga$, up to a Frobenius twist on $\Ga$. But \eqref{MofuGsequ} tells us that $\Hom_{\Gr A}(Q,\Ga) = 0$, hence $Q$ must be trivial.
\end{proof}

\begin{rem}
In particular, if $A=k$ is a field, then $\Grbfp qk$ is an abelian category, in which $\Ga$ is an injective object.

Moreover, in this case, we can drop the conditions that our group schemes $G$ are finitely presented and that our $k[F^r]$-modules are finitely generated, respectively. Indeed, we may argue as above, applying Theorem \ref{mainthmDG} on the way.

In fact, we can use \cite{DG}, III, $\mathsection$3, 7.5, in order to write $G = \lim G_i$ with $G_i \in \Grpa qk$ finitely presented, compatibly with $\M$ and $\G$. Namely, $G \to G_i$ is an epimorphism (cf. \textit{loc.cit.}, 7.4), hence $G_i$ inherits an $\FF_q$-action from $G$, by Remark \ref{categorical}. These agree in the limit,
\begin{equation}\label{limitaction}
  \begin{tikzcd}
    G \ar[r, two heads] \ar[d, "{[\alpha]}"] & G_j \ar[r, two heads] \ar[d, "{[\alpha]_j}"] & G_i \ar[d, "{[\alpha]_i}"] \\
    G \ar[r, two heads] & G_j \ar[r, two heads] & G_i
  \end{tikzcd}
\end{equation}
since the left square and the outer rectangle in \eqref{limitaction} commute, and so the two compositions in the right square agree up to $G \onto G_j$, hence must be equal.

Now choose an $\FF_q$-equivariant embedding $G \into \Ga^N$. Then we can take the pushout
\begin{equation}\label{pushoutembedding}
  \begin{tikzcd}
      G \ar[r, hook] \ar[d, two heads] \ar[dr, phantom, "{\pushout}" very near end] & \Ga^N \ar[d, two heads] \\
	  G_i \ar[r, hook] & H.
  \end{tikzcd}
\end{equation}
But then $H \isom \Ga^n$, by \cite{DG}, IV, $\mathsection$3, 6.8. To wit, $\M_p(H) \into \M_p(\Ga^N) = k[F]^{\oplus N}$ is injective by left-exactness, hence $\M_p(H) \isom k[F]^{\oplus n}$ for some $n$, since $k[F]$ is left-Euclidean. Therefore,
\[H = \G_p(\M_p(H)) \isom \Ga^n.\]
From Remark \ref{categorical}, we know that the diagram \eqref{pushoutembedding} in fact lies in $\Grpa qk$.
\end{rem}

Finally, let us prove the following structure theorem, generalizing Theorem \ref{structurethm}.

\begin{thm}\label{structurethmFq}
If $k$ is a perfect field, then $G\in\Grb qk$ lies in $\Grbfp qk$ if and only if
\[G \isom \Ga^n \times \pi_0(G) \times H,\]
with $H$ a product of group schemes of the form $\alpha_{q^s}$ and where $\pi_0(G)$ is an \'etale sheaf of finite-dimensional $\FF_q$-vector spaces. If $k$ is algebraically closed, then
\[\pi_0(G) \isom (\Fq)^m\]
for some $m\in\NN$.
\end{thm}
\begin{proof} As in \cite{DG}, IV, $\mathsection$3, 6.9, we use the fact that $k[F^r]$ is left-Euclidean to decompose $\M(G)$ into its $k[F^r]$-torsion submodule $M = \tors{\M(G)}$ and its torsionfree part. Furthermore, applying Lemma \ref{ssnil} to $M$ altogether yields the decomposition
\[G = \G(\M(G)) = \G(\M(G)/(M,f)) \times \G(\Mss,\fss) \times \G(\Mnil,\fnil).\]
Then $\M(G)/(M,f)$ is free of finite rank $n\in\NN$, and $\G(\M(G)/(M,f)) \isom \Ga^n$. The finite part of $G$ consists of the (by Lemma \ref{ssnil} maximal) \'etale part $\pi_0(G) = \G(\Mss,\fss)$ and the connected part $H = \G(\Mnil,\fnil)$, cf. Proposition \ref{properties}. By Theorem \ref{structurethm}, as a group scheme,
\[H \isom \alpha_{p^{s_1}} \times \ldots \times \alpha_{p^{s_h}} = \Spec(k[x_1,\ldots,x_h]/(x_1^{p^{s_1}},\ldots,x_h^{p^{s_h}})).\]
By Theorem \ref{frobsurj}, the maps $\Prim_1(B_H) \to \Prim_{p^s}(B_H),\ x \mapsto x^{p^s}$, are surjective. Thus $H$ can be written as above even in $\grb qk$. But then $r | s_i$ for all $i$, cf. Example \ref{balancedex}.

If $k$ is algebraically closed, then $\pi_0(G)$ is constant. Furthermore, it is killed by $p = F \circ V$ (cf. Remark \ref{FVpVF}) and of $q$-power order. Hence it is indeed a power of $\Fq = \Spec(k[x]/(x^q - x))$. Again as above, the $\FF_q$-action must be the canonical one.
\end{proof}

\begin{rem}
Note that $\underline{\FF_{p^t}}$ is of $\FF_q$-additive type if and only if $r|t$.
\end{rem}

\bigskip

\bibliography{grpschFq}
\bibliographystyle{plain}

\bigskip

\end{document}